\documentclass[12pt]{article}
\usepackage{amsfonts, amsmath, amssymb, latexsym, eucal, amscd} 

\usepackage[dvips]{pict2e}

\usepackage{amsthm}
\usepackage{amscd}

\usepackage[dvipdfmx]{graphics}
\usepackage{cite}

\newtheorem{definition}{Definition}[section]
\newtheorem{theorem}[definition]{Theorem}
\newtheorem{lemma}[definition]{Lemma}
\newtheorem{corollary}[definition]{Corollary}
\newtheorem{remark}[definition]{Remark}

\newtheorem{proposition}[definition]{Proposition}

\begin{document} 

\title{\bf Raising and lowering maps \\
for tridiagonal pairs
}
\author{
Paul Terwilliger
}
\date{}
\maketitle
\begin{abstract} Let $V$ denote a nonzero finite-dimensional vector space. A tridiagonal pair on $V$ is an ordered pair $A, A^*$ of maps in ${\rm End}(V)$
such that (i) each of $A, A^*$ is diagonalizable; (ii) there exists an ordering $\lbrace V_i \rbrace_{i=0}^d$ of the eigenspaces of $A$ such that
$A^* V_i \subseteq V_{i-1} + V_i + V_{i+1}$ $(0 \leq i \leq d)$, where $V_{-1} =0$ and $V_{d+1}=0$;
(iii) there exists an ordering $\lbrace V^*_i \rbrace_{i=0}^\delta$ of the eigenspaces of $A^*$ such that
$A V^*_i \subseteq V^*_{i-1} + V^*_i + V^*_{i+1}$ $(0 \leq i \leq \delta)$, where $V^*_{-1} =0$ and $V^*_{\delta+1}=0$;
(iv) there does not exist a subspace $W \subseteq V$ such that $W \not=0$, $W\not=V$, $A W \subseteq W$, $A^*W \subseteq W$.
Assume that $A, A^*$ is a tridiagonal pair on $V$. It is known that $d=\delta$.
For $0 \leq i \leq d$ let $\theta_i$ (resp. $\theta^*_i$) denote the eigenvalue of $A$ (resp. $A^*$) for $V_i$ (resp. $V^*_i$).
By construction, there exist $R,F,L \in {\rm End}(V)$ such that $A=R+F+L$ and $R V^*_i \subseteq V^*_{i+1}$, $F V^*_i \subseteq V^*_i$, $LV^*_i \subseteq V^*_{i-1}$ $(0 \leq i \leq d)$.
For $0 \leq i \leq d$ define $U_i = (V^*_0 + V^*_1 + \cdots + V^*_i ) \cap (V_i + V_{i+1} + \cdots + V_d)$. It is known that the sum $V=\sum_{i=0}^d U_i$
is direct. By construction, there exists $\mathcal R, \mathcal L \in {\rm End}(V)$ such that $\mathcal R=A - \theta_i I$ and $\mathcal L= A^*-\theta^*_i I$ on $U_i$ $(0 \leq i \leq d)$.
 It is known that $\mathcal R U_i \subseteq U_{i+1}$ and $\mathcal L U_i \subseteq U_{i-1}$ $(0 \leq i \leq d)$, where $U_{-1}=0$ and $U_{d+1}=0$.
 In this paper, our main goal is to describe how $R,F,L,\mathcal R, \mathcal L$ are related. We also give some results concerning  injectivity/surjectivity and $R, L$.
 \bigskip

 \end{abstract}
 
\section{Introduction}  This paper is about a linear algebraic object called a tridiagonal pair. The tridiagonal pair concept was formally introduced in \cite{TD00},
although it informally shows up earlier in connection with 
 the subconstituent algebra of a $Q$-polynomial distance-regular graph \cite{tSub1, tSub2,tSub3}.
Before we state our purpose, let us recall the definition of a tridiagonal pair.
\medskip

\noindent Let $\mathbb  F$ denote a field, and let $V$ denote a nonzero finite-dimensional vector space over $\mathbb F$. 
Let the algebra ${\rm End}(V)$ consist of the $\mathbb F$-linear maps from $V$ to $V$.
\begin{definition}
\label{def:TDP} 
\rm (See \cite[Definition~1.1]{TD00}.)
A {\it tridiagonal pair on $V$} is an ordered pair $A, A^*$ of maps in ${\rm End}(V)$
that satisfy the following conditions {\rm (i)--(iv)}:
\begin{enumerate}
\item[\rm (i)] each of $A, A^*$ is diagonalizable; 
\item[\rm (ii)] there exists an ordering $\lbrace V_i \rbrace_{i=0}^d$ of the eigenspaces of $A$ such that
\begin{align}
A^* V_i \subseteq V_{i-1} + V_i + V_{i+1}\qquad \qquad (0 \leq i \leq d),
\label{eq:TDA}
\end{align}
 where $V_{-1} =0$ and $V_{d+1}=0$;
\item[\rm (iii)] there exists an ordering $\lbrace V^*_i \rbrace_{i=0}^\delta$ of the eigenspaces of $A^*$ such that
\begin{align}
A V^*_i \subseteq V^*_{i-1} + V^*_i + V^*_{i+1}\qquad \qquad (0 \leq i \leq \delta),
\label{eq:TDAs}
\end{align}
 where $V^*_{-1} =0$ and $V^*_{\delta+1}=0$;
\item[\rm (iv)]
 there does not exist a subspace $W \subseteq V$ such that $W \not=0$, $W\not=V$, $A W \subseteq W$, $A^*W \subseteq W$.
\end{enumerate}
\end{definition}
\noindent For the rest of this section, let $A,A^*$ denote a tridiagonal pair on $V$ as in Definition \ref{def:TDP}. By \cite[Lemma~4.5]{TD00} the parameters $d, \delta$ from
Definition \ref{def:TDP} are equal;  
 their common value is called the {\it diameter} of the pair. 
 For $0 \leq i \leq d$ let $\theta_i$ (resp. $\theta^*_i$) denote the eigenvalue of $A$ (resp. $A^*$) for the eigenspace $V_i$ (resp. $V^*_i$).
 \medskip

\noindent 
In the study of $A, A^*$ there are several types of maps in ${\rm End}(V)$ that get called a raising map or lowering map.
Some examples are described in the next two paragraphs.
\medskip

\noindent  By Definition \ref{def:TDP}(iii), there exist $R,F,L \in {\rm End}(V)$ such that $A=R+F+L$
and 
\begin{align*}
R V^*_i \subseteq V^*_{i+1}, \qquad \quad F V^*_i \subseteq V^*_i, \qquad \quad LV^*_i \subseteq V^*_{i-1} \qquad \qquad (0 \leq i \leq d).
\end{align*}
It is natural to call $R,F,L$ a raising map, flat map, and lowering map, respectively. These maps come up naturally in the context of
distance-regular graphs \cite{uniform1, uniform2,dickie, dickie2, jtgo, uniformMT, worawannotai}.
In this context, the equation $A=R+F+L$ is sometimes called the quantum decomposition of $A$; see \cite[Definition~2.24]{hora} and \cite[Section~4]{zitnik}.
\medskip

\noindent  For $0\leq i \leq d$ define
\begin{align*}
U_i = (V^*_0 + V^*_1 + \cdots + V^*_i) \cap (V_i + V_{i+1} + \cdots + V_d).
\end{align*}
By \cite[Theorem~4.6]{TD00}  the sum $V=\sum_{i=0}^d U_i$ is direct. This sum is often called a split decomposition of $V$.
By \cite[Theorem~4.6]{TD00} we have
\begin{align}
&(A-\theta_i I ) U_i \subseteq U_{i+1} \qquad \quad (0 \leq i \leq d-1), \qquad \quad (A-\theta_d I) U_d =0, \label{eq:Araise}
\\
&(A^*-\theta^*_i I ) U_i \subseteq U_{i-1} \qquad \quad (1 \leq i \leq d), \qquad \quad (A^*-\theta^*_0 I) U_0 =0. \label{eq:Aslower}
\end{align}
By construction, there exist $\mathcal R, \mathcal L \in {\rm End}(V)$ such that for $0 \leq i \leq d$ the following hold on $U_i$:
\begin{align*}
\mathcal R = A - \theta_i I,\qquad \qquad \mathcal L = A^*- \theta^*_i I.
\end{align*}
By this and \eqref{eq:Araise}, \eqref{eq:Aslower} we have
\begin{align*}
&\mathcal R U_i \subseteq U_{i+1} \qquad \quad (0 \leq i \leq d-1), \qquad \quad \mathcal R U_d =0, 
\\
& \mathcal L U_i \subseteq U_{i-1} \qquad \quad (1 \leq i \leq d), \qquad \quad \mathcal L U_0 =0. 
\end{align*}
It is natural to call $\mathcal R$ (resp. $\mathcal L$) a raising map (resp. lowering map).
These maps have been used to (i)  obtain bounds on the shape vector of $A, A^*$ \cite{shape, nomShape}; (ii) turn $V$ into a module for the quantum group $U_q({\widehat{\mathfrak{sl}}}_2)$ 
\cite{tdanduqsl2hat, tdqrac};
(iii) construct the Drinfeld polynomial of $A, A^*$ \cite{TerDrinfeld, augIto}; (iv) construct the Bockting double-down operator $\psi$ \cite{twocom, bockting2, bockTer,bockting3}.
\medskip

\noindent In this paper, our main goal is to describe how $R,F,L, \mathcal R, \mathcal L$ are related. This description has the following features. In Theorem \ref{thm:ggee}, we give
some relations that involve $R,F,L$. In Theorems \ref{thm:m1}--\ref{thm:m3}, we use commuting diagrams to describe how $R,F,L$ are related to $\mathcal R, \mathcal L$.
In Theorems \ref{thm:rest}, \ref{thm:restDual} we give some relations that involve $\mathcal R, \mathcal L$.
\medskip

\noindent  We also have some results concerning  injectivity/surjectivity and $R, L$. To motivate these results we cite some prior work.  By \cite[Lemma~6.5]{TD00} the
following hold for $0 \leq i\leq j\leq d$:
\begin{enumerate}
\item[$\bullet$] the map ${\mathcal R}^{j-i}: U_i \to U_j$ is injective if $i+j\leq d$, bijective if $i+j=d$, and surjective if $i+j\geq d$;
\item[$\bullet$]  the map ${\mathcal L}^{j-i}: U_j \to U_i$ is injective if $i+j\geq d$, bijective if $i+j=d$, and surjective if $i+j\leq d$.
\end{enumerate}
In Theorem \ref{thm:bij2}, we show that the following hold for $0 \leq i\leq j\leq d$:
\begin{enumerate}
\item[$\bullet$] the map $R^{j-i}: V^*_i \to V^*_j$ is injective if $i+j\leq d$, bijective if $i+j=d$, and surjective if $i+j\geq d$;
\item[$\bullet$]  the map $L^{j-i}: V^*_j \to V^*_i$ is injective if $i+j\geq d$, bijective if $i+j=d$, and surjective if $i+j\leq d$.
\end{enumerate}
Near the end of the paper, we examine two special cases in extra detail. 
In the first special case, we  assume that $V_i$ and $V^*_i$ have dimension one for $0 \leq i \leq d$. In this case, the pair $A, A^*$ is
called a Leonard pair. In the second case, we assume that $\theta_i = d-2i$ and $\theta^*_i = d-2i$ for $0 \leq i \leq d$. In this case,
the pair $A, A^*$ is said to have Krawtchouk type.
\medskip

\noindent The paper is organized as follows. Section 2 contains some preliminaries.
In Section 3, we recall the concept of a tridiagonal system $\Phi$ and give some basic facts about it.
In Section 4, we define the maps $R,F,L$ and describe their basic properties.
In Section 5, we describe how $R,F,L$ are related.
In Section 6, we recall the $\Phi$-split decomposition of $V$.
In Section 7, we use the $\Phi$-split decomposition to define the maps $\mathcal R, \mathcal L$.
In Section 8, we describe how $R,F,L$ are related to $\mathcal R, \mathcal L$.
In Section 9, we describe how $\mathcal R, \mathcal L$ are related.
In Section 10, we give some results involving injectivity/surjectivity and $R, L$.
In Section 11, we  describe what our earlier
results become under the assumption that $\Phi$ is a Leonard system.
In Section 12, we  describe what our earlier
results become under the assumption that $\Phi$ has Krawtchouk type.

\section{Preliminaries}
We now begin our formal argument. Throughout the paper, we will use
the following concepts and notational conventions.
 For an integer $d\geq 0$ a sequence $x_0, x_1, \ldots, x_d$ will be denoted by $\lbrace x_i \rbrace_{i=0}^d$.
Let $\mathbb F$ denote a field. 
Every vector space that we encounter is understood to be over $\mathbb F$. Every algebra that we encounter is 
understood to be associative, over $\mathbb F$, and have a multiplicative identity.
A subalgebra has the same multiplicative identity as the parent algebra.
Throughout the paper, $V$ denotes a nonzero  finite-dimensional vector space. The algebra ${\rm End}(V)$ 
consists of the $\mathbb F$-linear maps from $V$ to $V$. Let $I \in {\rm End}(V)$ denote the identity map.
An element $A \in {\rm End}(V)$ is said to be {\it diagonalizable} whenever
$V$ is spanned by the eigenspaces of $A$. Assume that $A$ is diagonalizable, and let $\lbrace V_i \rbrace_{i=0}^d$ 
denote an ordering of the eigenspaces of $A$. By linear algebra, the sum $V=\sum_{i=0}^d V_i$ is direct. For $0 \leq i \leq d$ let $\theta_i$ denote the eigenvalue of $A$ for $V_i$.
For $0 \leq i \leq d$ define $E_i \in {\rm End}(V)$ such that
$(E_i-I) V_i=0$ and $E_iV_j=0$ if $j \not=i$ $(0 \leq j \leq d)$. Thus $E_i$ is the projection from $V$ onto $V_i$. We call $E_i$ the {\it primitive idempotent} of $A$ associated with $V_i$ (or $\theta_i$).
We have:
(i) $E_i E_j = \delta_{i,j} E_i$ $ (0 \leq i,j\leq d)$;
(ii) $I = \sum_{i=0}^d E_i$; 
(iii) $V_i = E_iV$ $ (0 \leq i \leq d)$; 
(iv) $A = \sum_{i=0}^d \theta_i E_i$; 
(v) $AE_i = \theta_i E_i = E_iA$ $(0 \leq i \leq d)$. Moreover
\begin{align*}
  E_i=\prod_{\stackrel{0 \leq j \leq d}{j \neq i}}
          \frac{A-\theta_jI}{\theta_i-\theta_j} \qquad \qquad (0 \leq i \leq d).
\end{align*}
Let $M$ denote the subalgebra of ${\rm End}(V)$ generated by $A$. 
The vector space $M$ has a basis $\lbrace A^i \rbrace_{i=0}^d $ and a basis $\lbrace E_i \rbrace_{i=0}^d$.
For elements $X,Y$ in any algebra, their commutator is $\lbrack X,Y \rbrack = XY-YX$.

\section{Tridiagonal systems}
In our discussion of tridiagonal pairs, it is convenient to bring in the concept of a tridiagonal system. We will define this concept shortly. 
 Let $A, A^*$ denote a tridiagonal pair on $V$.
An ordering of the eigenspaces of $A$ (resp. $A^*$)
is said to be
{\it standard} whenever it satisfies
(\ref{eq:TDA}) (resp. 
(\ref{eq:TDAs})). 
Let 
 $\lbrace V_i \rbrace_{i=0}^d$ denote a standard ordering
of the eigenspaces of $A$.
By \cite[Lemma~2.4]{TD00}, the inverted ordering 
 $\lbrace V_{d-i} \rbrace_{i=0}^d$ 
is also standard and no further ordering is standard.
A similar result holds for the 
 eigenspaces of $A^*$. An ordering of the primitive idempotents of $A$ (resp. $A^*$) is said to be {\it standard} whenever
 the corresponding ordering of the eigenspaces of $A$ (resp. $A^*$) is standard.

\begin{definition} \rm (See  \cite[Definition~2.1]{TD00}, \cite[Definition~2.1]{nomsharp}.)
 \label{def:TDsystem} 
A  {\it tridiagonal system on $V$} is a sequence
\begin{align*}
 \Phi=(A;\lbrace E_i\rbrace_{i=0}^d;A^*;\lbrace E^*_i\rbrace_{i=0}^d)
\end{align*}
that satisfies the following three conditions:
\begin{enumerate}
\item[\rm (i)]
$A,A^*$ is a tridiagonal pair on $V$;
\item[\rm (ii)]
$\lbrace E_i\rbrace_{i=0}^d$ is a standard ordering
of the primitive idempotents of $A$;
\item[\rm (iii)]
$\lbrace E^*_i\rbrace _{i=0}^d$ is a standard ordering
of the primitive idempotents of $A^*$.
\end{enumerate}
\end{definition}

\noindent For the rest of this paper, we fix a tridiagonal system $\Phi = ( A; \lbrace E_i \rbrace_{i=0}^d; A^*; \lbrace E^*_i \rbrace_{i=0}^d )$ on $V$.
To avoid trivialities, we always assume that $d\geq 1$.
\medskip

\noindent For the rest of this section, we review some basic facts about $\Phi$. More information can be found in \cite[Chapter~6.2]{bbit} and \cite{class,TD00,augIto,nomTB, nomTowards, qSerre, madrid, terAlt, int}.
\begin{lemma} {\rm (See \cite[Lemma~3.6]{terAlt}.)}
\label{lem:EAEp}
The following hold for $0 \leq i,j \leq d$:
\begin{enumerate}
\item[\rm (i)]  $E^*_i AE^*_j = \begin{cases} 0 & {\rm if}\; \vert i - j \vert >1, \\
                                                                      \not=0 & {\rm if} \; \vert i - j \vert = 1
                                                                      \end{cases}$
\item[\rm (ii)]   $E_i A^* E_j = \begin{cases} 0 & {\rm if}\; \vert i - j \vert >1, \\
                                                                      \not=0 & {\rm if} \; \vert i - j \vert = 1.
                                                                      \end{cases}$
\end{enumerate}
\end{lemma}

\noindent We recall the shape of $\Phi$. By \cite[Corollary~5.7]{TD00}, for $0 \leq i \leq d$ the subspaces $E_iV$, $E_{d-i}V$, $E^*_iV$, $E^*_{d-i}V$ have the same dimension; 
we denote this common dimension by $\rho_i$. By construction $\rho_i \not=0$ and $\rho_i = \rho_{d-i}$. By  \cite[Corollary~6.6]{TD00} we have
\begin{align}
\rho_{i-1} \leq \rho_i \qquad \qquad  (1 \leq i \leq d/2). \label{eq:unimodal}
\end{align}
The sequence $\lbrace \rho_i \rbrace_{i=0}^d$ is called the {\it shape} of $\Phi$.
\medskip

\begin{definition}
\label{def:eigenval}
\rm (See \cite[Definition~3.1]{TD00}.)
 For $0 \leq i \leq d$ let $\theta_i$ (resp. $\theta^*_i$) denote the eigenvalue of $A$ (resp. $A^*$) for $E_i$ (resp. $E^*_i$).
 By construction, the scalars $\lbrace \theta_i \rbrace_{i=0}^d$ are contained in $\mathbb F$ and mutually distinct.
 Similarly, the scalars $\lbrace \theta^*_i \rbrace_{i=0}^d$ are contained in $\mathbb F$ and mutually distinct.
 We call the sequence $\lbrace \theta_i \rbrace_{i=0}^d$ (resp. $\lbrace \theta^*_i \rbrace_{i=0}^d$) the
 {\it eigenvalue sequence} (resp. {\it dual eigenvalue sequence}) of $\Phi$. 
\end{definition}

\noindent By construction,
\begin{align}
&A  = \sum_{i=0}^d \theta_i E_i, \qquad \qquad A^*= \sum_{i=0}^d \theta^*_i E^*_i.
\label{eq:AAs}
\end{align}

\begin{lemma} 
{\rm (See \cite[Theorem~11.1]{TD00}.)}
The scalars
\begin{align*}
\frac {\theta_{i-2}-\theta_{i+1}}{\theta_{i-1}-\theta_i}, \qquad \qquad 
 \frac{\theta^*_{i-2}-\theta^*_{i+1}}{\theta^*_{i-1}-\theta^*_i}
\end{align*}
are equal and independent of $i$ for $2 \leq i \leq d-1$.
\end{lemma}

\noindent Next we recall the tridiagonal relations.
\begin{lemma} \label{lem:fiveP}
{\rm (See  \cite[Theorem~10.1]{TD00}.)} 
There exists a sequence $\beta, \gamma, \gamma^*, \varrho, \varrho^*$ of scalars in $\mathbb F$ such that
\begin{align}
0 &=\lbrack A,A^2A^*-\beta AA^*A + 
A^*A^2 -\gamma (AA^*+A^*A)-\varrho A^*\rbrack,           \label{eq:TD1}
\\
0 &= \lbrack A^*,A^{*2}A-\beta A^*AA^* + AA^{*2} -\gamma^* (A^*A+AA^*)-
\varrho^* A\rbrack. \label{eq:TD2}
\end{align}
The sequence  $\beta, \gamma, \gamma^*, \varrho, \varrho^*$ is unique if $d\geq 3$.
\end{lemma}
\noindent The relations \eqref{eq:TD1}, \eqref{eq:TD2} are called the {\it tridiagonal relations}; see \cite[Section~3]{qSerre}.

\begin{remark}
\label{rem:expand}
\rm The tridiagonal relations \eqref{eq:TD1}, \eqref{eq:TD2} look as follows in expanded form:
\begin{align}
\begin{split} 
&A^3 A^* - (\beta+1) A^2 A^* A + (\beta+1) A A^* A^2 - A^* A^3  \\
& \quad = \gamma (A^2 A^*-A^* A^2) + \varrho (A A^*-A^* A), 
\end{split} \label{eq:TD1Ex} \\
\begin{split}
& A^{*3} A - (\beta+1) A^{*2} A A^* + (\beta+1) A^* A A^{*2} - A A^{*3} \\
& \quad = \gamma^* (A^{*2} A-A A^{*2}) + \varrho^* (A^* A - A A^*).
 \end{split} \label{eq:TD2Ex}
\end{align}
\end{remark}

\noindent In the next result, we explain how the sequence $\beta, \gamma, \gamma^*, \varrho, \varrho^*$ from Lemma \ref{lem:fiveP} is
related to the eigenvalue sequence and dual eigenvalue sequence of $\Phi$.

\begin{lemma} \label{lem:beta} 
{\rm (See \cite[Theorem~11.1]{TD00}.)}
Let $\beta, \gamma, \gamma^*, \varrho, \varrho^* \in \mathbb F$.
The sequence  $\beta, \gamma, \gamma^*, \varrho, \varrho^* $ satisfies 
 \eqref{eq:TD1}, \eqref{eq:TD2}
 if and only if the following {\rm (i)--(iii)} hold.
\begin{enumerate}
\item[\rm (i)] For $2 \leq i \leq d-1$,
\begin{align}
\beta+1&=\frac {\theta_{i-2}-\theta_{i+1}}{\theta_{i-1}-\theta_i}= 
 \frac{\theta^*_{i-2}-\theta^*_{i+1}}{\theta^*_{i-1}-\theta^*_i}.       \label{eq:beta}
 \end{align}
 \item[\rm (ii)] For $1 \leq i \leq d-1$,
\begin{align} 
\gamma &= \theta_{i-1}-\beta \theta_i + \theta_{i+1},  \label{eq:gamma}
\\
\gamma^* &= \theta^*_{i-1}-\beta \theta^*_i + \theta^*_{i+1}.  \label{eq:gammas}
\end{align}
\item[\rm (iii)]
For $1 \leq i \leq d$,
\begin{align*}
\varrho&= \theta^2_{i-1}-\beta \theta_{i-1}\theta_i+\theta_i^2-\gamma (\theta_{i-1}+\theta_i),       
\\
\varrho^*&= \theta^{*2}_{i-1}-\beta \theta^*_{i-1}\theta^*_i+\theta_i^{*2}-
\gamma^* (\theta^*_{i-1}+\theta^*_i).           
\end{align*}
\end{enumerate}
\end{lemma}

\noindent For the rest of the paper, we fix a sequence $\beta, \gamma, \gamma^*, \varrho, \varrho^*$ of scalars in $\mathbb F$
that satisfies  \eqref{eq:TD1}, \eqref{eq:TD2}.
\medskip

\noindent The following definition is for notational convenience.

\begin{definition}
\label{def:extend}
\rm Define the scalars $\theta_{-1}$, $\theta_{d+1}$, $\theta^*_{-1}$, $\theta^*_{d+1}$ such that  \eqref{eq:gamma}, \eqref{eq:gammas} hold at $i=0$ and $i=d$.
\end{definition}

\section{The maps $R,F,L$}
We continue to discuss the tridiagonal system $\Phi = ( A; \lbrace E_i \rbrace_{i=0}^d; A^*; \lbrace E^*_i \rbrace_{i=0}^d )$ on $V$.
In this section, we use $\Phi$ to define some maps $R, F, L \in {\rm End}(V)$. We describe some basic properties of $R,F,L$.
\medskip

\noindent The following definition is motivated by \cite[Section~1.5]{dickie}.

\begin{definition}
\label{def:RFL} 
\rm Define
\begin{align*}
R = \sum_{i=0}^{d-1} E^*_{i+1} A E^*_i, \qquad \quad
F= \sum_{i=0}^d E^*_i A E^*_i, \qquad \quad
L = \sum_{i=1}^d E^*_{i-1} A E^*_i.
\end{align*}
We call $R,F, L$ the {\it $\Phi$-raising map}, the {\it $\Phi$-flat map}, and the {\it $\Phi$-lowering map}, respectively.
\end{definition}

\begin{lemma}
\label{lem:ARFL}
We have
\begin{align*}
A = R + F + L.
\end{align*}
\end{lemma}
\begin{proof} By Lemma
\ref{lem:EAEp}
and Definition \ref{def:RFL}.
\end{proof}

\begin{lemma}
\label{lem:RFLmeaning}
We have
\begin{align*}
&R E^*_iV \subseteq E^*_{i+1}V \quad (0 \leq i \leq d-1), \qquad \quad R E^*_dV=0, \\
&F E^*_iV \subseteq E^*_iV \quad (0 \leq i \leq d), \\
&L E^*_iV \subseteq E^*_{i-1}V \quad (1 \leq i \leq d), \qquad \quad LE^*_0V=0.
\end{align*}
\end{lemma}
\begin{proof} By Definition 
\ref{def:RFL}.
\end{proof}

\begin{lemma} \label{lem:RLnil}
For $0 \leq i \leq d$ we have
\begin{align*}
  R^{d-i+1} E^*_i =0, \qquad \qquad         L^{i+1} E^*_i =0.
\end{align*}
\noindent Moreover,
\begin{align*}
R^{d+1}=0,\qquad \qquad L^{d+1}=0.
\end{align*}
\end{lemma}
\begin{proof} By Lemma \ref{lem:RFLmeaning}.
\end{proof}

\begin{lemma}
\label{lem:EsRFL}
We have
\begin{align*}
&E^*_{i+1}AE^*_i = R E^*_i = E^*_{i+1} R \qquad \qquad (0 \leq i \leq d-1), \\
&E^*_i A E^*_i = F E^*_i = E^*_i F \qquad \qquad (0 \leq i \leq d), \\
&E^*_{i-1} A E^*_i = L E^*_i = E^*_{i-1} L \qquad \qquad (1 \leq i \leq d).
\end{align*}
\end{lemma}
\begin{proof} By Definition 
\ref{def:RFL}.
\end{proof}

\section{How $R,F,L$ are related}
We continue to discuss the tridiagonal system $\Phi = ( A; \lbrace E_i \rbrace_{i=0}^d; A^*; \lbrace E^*_i \rbrace_{i=0}^d )$ on $V$.
In this section, we use the tridiagonal relations to obtain some relations involving $R,F,L$. 
\medskip

\noindent The following result
is motivated by \cite[Lemmas~5.5,~5.6]{tSub3}.

\begin{theorem} 
\label{thm:ggee}
The following {\rm (i)--(iii)} hold.
\begin{enumerate}
\item[\rm (i)] For $2 \leq i \leq d$ the equation
\begin{align}
g^-_i FL^2 + LFL + g^+_i L^2 F = \gamma L^2         \label{eq:LLF}
\end{align}
holds on $E^*_iV$ and the equation
\begin{align}
g^-_i R^2 F  + RFR + g^+_i FR^2  = \gamma R^2        \label{eq:RRF}
\end{align}
holds on $E^*_{i-2}V$,
where
\begin{align}
g^+_i  = \frac{\theta^*_i - \theta^*_{i+1}}{\theta^*_i - \theta^*_{i-2}}, \qquad \qquad
g^-_i = \frac{\theta^*_{i-2} - \theta^*_{i-3}}{ \theta^*_{i-2} - \theta^*_i}. 
\label{eq:gpm}
\end{align}
\item[\rm (ii)] For $1 \leq i \leq d$ the equation
\begin{align*}
e^-_i RL^2 + (\beta+2) LRL + e^+_i L^2 R + LF^2 -\beta FLF  + F^2 L = \gamma (LF+FL) + \varrho L
\end{align*}
holds on $E^*_iV$ and the equation
\begin{align*}
e^-_i R^2 L + (\beta+2) RLR + e^+_i LR^2 + F^2R -\beta FRF  + RF^2  = \gamma (FR+RF) + \varrho R
\end{align*}
holds on $E^*_{i-1}V$,
where
\begin{align}
e^+_i = \frac{\theta^*_i - \theta^*_{i+2}}{\theta^*_i - \theta^*_{i-1}} \quad  (1 \leq i \leq d-1), \qquad \quad 
e^-_i = \frac{\theta^*_{i-1}-\theta^*_{i-3}}{\theta^*_{i-1} - \theta^*_i} \quad  (2 \leq i \leq d)
\label{eq:epm}
\end{align}
and $e^+_d, e^-_1$ are indeterminates.
\item[\rm (iii)]  For $0 \leq i \leq d$ the equation
\begin{align*}
(\theta^*_i-\theta^*_{i+1}) \lbrack F, LR \rbrack = (\theta^*_{i-1} - \theta^*_i) \lbrack F, RL\rbrack
\end{align*}
holds on $E^*_iV$.
\end{enumerate}
\end{theorem}
\begin{proof} (i) We first check \eqref{eq:LLF}. Let $C$ denote the left-hand side of \eqref{eq:TD1Ex} minus
the right-hand side of \eqref{eq:TD1Ex}. Thus $C=0$, so 
\begin{align}
0 = E^*_{i-2} C E^*_i. 
\label{eq:C}
\end{align}
We evaluate the terms in \eqref{eq:C}. Observe that 
the following hold on $E^*_iV$:
\begin{align*}
E^*_{i-2} A^3 A^* E^*_i &=  \theta^*_i FL^2  +    \theta^*_i L F L      +   \theta^*_i L^2 F,     \\
E^*_{i-2} A^2 A^* AE^*_i &=  \theta^*_{i-1} FL^2  +    \theta^*_{i-1} L F L      +   \theta^*_i L^2 F,     \\
E^*_{i-2} A A^* A^2 E^*_i &=  \theta^*_{i-2} FL^2  +    \theta^*_{i-1} L F L      +   \theta^*_{i-1} L^2 F,     \\
E^*_{i-2} A^* A^3  E^*_i &=  \theta^*_{i-2} FL^2  +    \theta^*_{i-2} L F L      +   \theta^*_{i-2} L^2 F,     \\
E^*_{i-2} A^2 A^* E^*_i &= \theta^*_i L^2, \qquad \qquad  E^*_{i-2} A^* A^2  E^*_i = \theta^*_{i-2} L^2,  \\
E^*_{i-2} A A^* E^*_i &=0, \qquad \qquad \qquad E^*_{i-2} A^* A E^*_i=0.
\end{align*}
Evaluating \eqref{eq:C} 
using the above comments,  we routinely obtain \eqref{eq:LLF}.
We obtain \eqref{eq:RRF} in a similar manner. \\
\noindent (ii), (iii) Similar to the proof of (i) above.
\end{proof}

\begin{lemma} 
\label{lem:ggee}
We have
\begin{align*}
g^+_i &\not=0 \quad  (2 \leq i \leq d-1),  \qquad \qquad 
g^-_i \not=0  \quad (3 \leq i \leq d),  \\
e^+_i &\not=0 \quad  (1 \leq i \leq d-2), \qquad \qquad 
e^-_i \not=0  \quad (3 \leq i \leq d).
\end{align*}
\end{lemma}
\begin{proof} By \eqref{eq:gpm}, \eqref{eq:epm} and since $\lbrace \theta^*_\ell\rbrace_{\ell=0}^d$ are mutually distinct.
\end{proof}

\section{The $\Phi$-split decomposition  of $V$}
We continue to discuss the tridiagonal system $\Phi = ( A; \lbrace E_i \rbrace_{i=0}^d; A^*; \lbrace E^*_i \rbrace_{i=0}^d )$ on $V$.
In this section we recall the $\Phi$-split decomposition of $V$. Background information about this topic can be found in 
\cite[Sections~4--6]{TD00} and \cite{twocom, bockting2, bockTer, bockting3,  shape, tdanduqsl2hat, tdqrac, augIto, nomSplit, nomShape, nomsharp, terAlt}.
\medskip

\noindent For $0 \leq i \leq d$ define
\begin{align*}
U_i = \bigl( E^*_0V + E^*_1V+\cdots + E^*_iV \bigr) \cap \bigl( E_iV + E_{i+1}V + \cdots + E_dV\bigr).
\end{align*}
By \cite[Theorem~4.6]{TD00}  the sum $V=\sum_{i=0}^d U_i$ is direct.
Also by \cite[Theorem~4.6]{TD00}, 
\begin{align} \label{eq:usum1}
&U_0 + U_1 + \cdots + U_i = E^*_0V + E^*_1V + \cdots + E^*_iV \qquad \qquad (0 \leq i \leq d), \\
&U_i + U_{i+1} + \cdots + U_d = E_i V + E_{i+1}V + \cdots + E_dV \qquad \qquad (0 \leq i \leq d). \label{eq:usum2}
\end{align}
By \cite[Corollary~5.7]{TD00} the subspace $U_i$ has dimension $\rho_i$ for $0 \leq i \leq d$.
We call the sequence $\lbrace U_i \rbrace_{i=0}^d$ the {\it $\Phi$--split decomposition of $V$}.
By \cite[Theorem~4.6]{TD00}  we have 
\begin{align*}
&(A-\theta_i I ) U_i \subseteq U_{i+1} \qquad (0 \leq i \leq d-1), \qquad \quad (A-\theta_d I)U_d=0, \\
&(A^*-\theta^*_i I ) U_i \subseteq U_{i-1} \qquad (1 \leq i \leq d), \qquad \quad (A^*-\theta^*_0 I)U_0=0.
\end{align*}

\begin{definition}
\label{def:Fi}
\rm (See \cite[Definition~5.2]{TD00}.)
For $0 \leq i \leq d$ define $F_i \in {\rm End}(V)$ such that $(F_i - I ) U_i=0$ 
and $F_i U_j = 0$ if $j \not=i $ $(0 \leq j \leq d)$.
\end{definition}

\begin{lemma}
\label{lem:factsFi}
{\rm (See \cite[Lemma~5.3]{TD00}.)} 
We have $U_i = F_iV$ for $0 \leq i \leq d$. Moreover,
\begin{align*}
 F_i F_j = \delta_{i,j} F_i \quad (0 \leq i,j\leq d), \qquad \qquad I = \sum_{i=0}^d F_i.
\end{align*}
\end{lemma}

\noindent The next three lemmas are well known; we give short proofs for the sake of completeness.
\begin{lemma}
\label{lem:FE}
{\rm (See \cite[Lemma~5.4]{TD00}.)} We have
\begin{align*}
F_j E^*_i  = 0, \qquad \qquad E^*_j F_i=0, \qquad \qquad (0 \leq i < j \leq d).
\end{align*}
\end{lemma}
\begin{proof} We have $F_j E^*_i=0$ because
\begin{align*}
F_j E^*_iV & \subseteq F_j (E^*_0V+E^*_1V+\cdots + E^*_iV) \\
                 &= F_j (F_0V + F_1V + \cdots + F_iV) \\
                 &=0.
                 \end{align*}
 We have $E^*_j F_i=0$ because
\begin{align*}
 E^*_j F_iV & \subseteq E^*_j  (F_0V+ F_1V + \cdots + F_iV) \\
                 &= E^*_j (E^*_0V+E^*_1V+\cdots + E^*_iV) \\
                                 &=0.
                 \end{align*}
\end{proof}

\begin{lemma}
\label{lem:FEF}
{\rm (See \cite[Lemma~5.5]{TD00}.)} We have
\begin{align*}
F_i E^*_i F_i= F_i, \qquad \qquad E^*_i F_i E^*_i =E^*_i, \qquad \qquad (0 \leq i \leq d).
\end{align*}
\end{lemma}
\begin{proof} We first obtain $F_i E^*_i F_i = F_i$. We have
\begin{align*}
F_i = F^2_i = F_i I F_i = \sum_{j=0}^d F_i E^*_j F_i.
\end{align*}
For $0 \leq j \leq d$ with $i \not=j$,  we have $F_i E^*_j=0$ if $i>j$ and $E^*_j F_i =0$ if $j>i$. Therefore $F_i E^*_j F_i=0$.
By these comments, $F_i E^*_i F_i = F_i$. The equation $E^*_i F_i E^*_i =E^*_i$ is similarly obtained.
\end{proof}

\begin{lemma} 
\label{lem:Bij}
{\rm (See \cite[Lemma~3.15]{terAlt}.)}
The following {\rm (i)--(iii)} hold.
\begin{enumerate}
\item[\rm (i)] The map $\sum_{\ell=0}^d F_\ell E^*_\ell: V \to V$ is a bijection that sends
\begin{align*}
E^*_iV \mapsto F_i V \qquad \qquad (0 \leq i \leq d).
\end{align*}
\item[\rm (ii)] 
The map $\sum_{\ell=0}^d E^*_\ell F_\ell: V \to V$ is a bijection that sends
\begin{align*}
F_iV \mapsto E^*_i V \qquad \qquad (0 \leq i \leq d).
\end{align*}
\item[\rm (iii)] The maps 
$\sum_{\ell=0}^d F_\ell E^*_\ell $ and $\sum_{\ell=0}^d E^*_\ell F_\ell$ are inverses.
\end{enumerate}
\end{lemma}
\begin{proof} It is convenient to prove (iii) first. \\
\noindent (iii) Observe that
\begin{align*}
\Biggl( \sum_{\ell=0}^d F_\ell E^*_\ell \Biggr) \Biggl( \sum_{k=0}^d E^*_k F_k\Biggr)
= \sum_{\ell=0}^d F_\ell E^*_\ell F_\ell = \sum_{\ell=0}^d F_\ell = I
\end{align*}
and also 
\begin{align*}
\Biggl( \sum_{k=0}^d E^*_k F_k \Biggr) \Biggl( \sum_{\ell=0}^d F_\ell E^*_\ell \Biggr)
= \sum_{\ell=0}^d E^*_\ell F_\ell E^*_\ell = \sum_{\ell=0}^d E^*_\ell = I.
\end{align*}
\noindent (i), (ii) By (iii) and the construction.
\end{proof}

\begin{corollary}
\label{cor:rFE}
For $0 \leq i \leq d$ we have
\begin{align}
&{\rm rank}\, F_i E^*_i ={\rm rank}\, E^*_i F_i=\rho_i, \label{eq:rFEs} \\
&{\rm rank}\, F_i E_i  ={\rm rank}\, E_i F_i=\rho_i.   \label{eq:rFE}
\end{align}
\end{corollary}
\begin{proof} Concerning $F_iE^*_i$, 
we have $F_i E^*_i V = F_iV$ by Lemma \ref{lem:Bij}(i), so
 ${\rm dim} \, F_i E^*_i V = {\rm dim} \,F_i V = \rho_i$. Therefore, ${\rm rank}\,F_iE^*_i=\rho_i$. Similarly, we obtain  ${\rm rank}\,E^*_iF_i=\rho_i$.
 We have shown \eqref{eq:rFEs}. To get \eqref{eq:rFE}, apply \eqref{eq:rFEs} to the tridiagonal system 
$\Phi^\times = (A^*; \lbrace E^*_{d-\ell} \rbrace_{\ell=0}^d; A; \lbrace E_{d-\ell} \rbrace_{\ell=0}^d)$ and use $\rho_i = \rho_{d-i}$.
\end{proof}

\section{The maps $\mathcal R, \mathcal L$}
We continue to discuss the tridiagonal system $\Phi = ( A; \lbrace E_i \rbrace_{i=0}^d; A^*; \lbrace E^*_i \rbrace_{i=0}^d )$ on $V$.
In this section, we use the $\Phi$-split decomposition of $V$ to define some maps $\mathcal R, \mathcal L \in {\rm End}(V)$. These maps
were  introduced in \cite{TD00}. In the next section, we will describe how $R,F,L$ are related to $\mathcal R, \mathcal L$.

\begin{definition}
\label{def:RRLL}
\rm (See \cite[Definition~6.1]{TD00}.)
We define
\begin{align*}
{\mathcal R} = A - \sum_{i=0}^d \theta_i F_i, \qquad \qquad {\mathcal L} = A^*-\sum_{i=0}^d \theta^*_i F_i
\end{align*}
where $\lbrace F_i \rbrace_{i=0}^d$ are from Definition \ref{def:Fi}. We call $\mathcal R $ (resp. $ \mathcal L$) the
{\it $\Phi$-split raising map} (resp. {\it  $\Phi$-split lowering map}).
\end{definition}

\begin{lemma}
\label{lem:RRact}
{\rm (See \cite[Lemma~6.2]{TD00}.)}
For $0 \leq i \leq d$ the following hold on $F_iV$:
\begin{align*}
{\mathcal R} = A-\theta_i I, \qquad \qquad {\mathcal L} = A^*-\theta^*_i I.
\end{align*}
\end{lemma}

\begin{lemma}
\label{lem:RRmeaning}
{\rm (See \cite[Corollary~6.3]{TD00}.)} 
We have
\begin{align*}
&{\mathcal R} F_iV \subseteq F_{i+1}V \quad (0 \leq i \leq d-1), \qquad \quad {\mathcal R} F_dV = 0, \\
&{\mathcal L} F_i V\subseteq F_{i-1}V \quad (1 \leq i \leq d), \qquad \quad {\mathcal L} F_0V =0.
\end{align*}
\end{lemma}

\begin{lemma} \label{lem:cRLnil}
For $0 \leq i \leq d$ we have
\begin{align*}
  \mathcal R^{d-i+1} F_i =0, \qquad \qquad        \mathcal L^{i+1} F_i =0.
\end{align*}
Moreover,
\begin{align*}
\mathcal R^{d+1}=0, \qquad \qquad \mathcal L^{d+1}=0.
\end{align*}
\end{lemma}
\begin{proof} By Lemma \ref{lem:RRmeaning}.
\end{proof}

\begin{lemma}
\label{lem:RRmeaning2}
{\rm (See \cite[Lemmas~6.4, 7.1]{TD00}.)}
For $0 \leq i,j\leq d$ we have $F_i A F_j =0$ unless $i-j \in \lbrace 0,1\rbrace$.
Moreover
\begin{align*}
&F_{i+1} A F_i = {\mathcal R} F_i= F_{i+1} {\mathcal R} \quad (0 \leq i \leq d-1),  \\
& F_i A F_i = \theta_i F_i \qquad \qquad (0 \leq i \leq d).
\end{align*}
\end{lemma}

\begin{lemma}
\label{lem:LLmeaning2}
{\rm (See \cite[Lemmas~6.4, 7.2]{TD00}.)}
For $0 \leq i,j\leq d$ we have $F_i A^* F_j =0$ unless $j-i \in \lbrace 0,1\rbrace$.
Moreover
\begin{align*}
&F_{i-1} A^* F_i = {\mathcal L} F_i = F_{i-1} {\mathcal L} \qquad \qquad (1 \leq i \leq d), \\
& F_i A^* F_i = \theta^*_i F_i \qquad \qquad (0 \leq i \leq d).
\end{align*}
\end{lemma}

\begin{lemma}
\label{lem:RRbij}
{\rm (See \cite[Lemma~6.5]{TD00}.)}
For $0 \leq i\leq j\leq d$ the following {\rm (i), (ii)} hold.
\begin{enumerate}
\item[\rm (i)] The map ${\mathcal R}^{j-i}: F_iV \to F_jV$ is injective if $i+j\leq d$, bijective if $i+j=d$, and surjective if $i+j\geq d$.
\item[\rm (ii)]  The map ${\mathcal L}^{j-i}: F_jV \to F_iV$ is injective if $i+j\geq d$, bijective if $i+j=d$, and surjective if $i+j\leq d$.
\end{enumerate}
\end{lemma}

\section{How $R,F,L$ are related to $\mathcal R, \mathcal L$}
We continue to discuss the tridiagonal system $\Phi = ( A; \lbrace E_i \rbrace_{i=0}^d; A^*; \lbrace E^*_i \rbrace_{i=0}^d )$ on $V$.
In this section, we describe how $R,F,L$ are related to $\mathcal R, \mathcal L$.

\begin{lemma} 
\label{lem:FEstep1}
For $0 \leq i < j \leq d$,
\begin{align*}
F_i E^*_j = \frac{\mathcal L}{\theta^*_j - \theta^*_i} F_{i+1} E^*_j.
\end{align*}
\end{lemma}
\begin{proof} Observe that
\begin{align*}
\mathcal L F_{i+1} E^*_j = F_i \mathcal L E^*_j = F_i \Biggl( A^* - \sum_{\ell=0}^d \theta^*_\ell F_\ell \Biggr) E^*_j = (\theta^*_j - \theta^*_i) F_i E^*_j.
\end{align*}
The result follows.
\end{proof}

\begin{proposition}
\label{prop:FEform}
For $0 \leq i \leq j \leq d$,
\begin{align*}
F_i E^*_j = \frac{ {\mathcal L}^{j-i}}{(\theta^*_j-\theta^*_i)(\theta^*_j-\theta^*_{i+1}) \cdots (\theta^*_j - \theta^*_{j-1})} F_j E^*_j.
\end{align*}
\end{proposition}
\begin{proof} By Lemma \ref{lem:FEstep1} and induction on $j-i$.
\end{proof}

\begin{lemma}
\label{lem:EFstep1}
For $0 \leq i < j \leq d$,
\begin{align*}
E^*_i F_j = E^*_i F_{j-1} \frac{ \mathcal L}{ \theta^*_i - \theta^*_j}.
\end{align*}
\end{lemma}
\begin{proof}
Observe that
\begin{align*}
E^*_i F_{j-1} \mathcal L = E^*_i \mathcal L F_j = E^*_i \Biggl( A^*- \sum_{\ell=0}^d \theta^*_\ell F_\ell \Biggr) F_j = (\theta^*_i - \theta^*_j) E^*_i F_j.
\end{align*}
The result follows.
\end{proof}

\begin{proposition} 
\label{prop:EFform}
For $0 \leq i \leq j \leq d$,
\begin{align*}
E^*_i F_j = E^*_i F_i \frac{ {\mathcal L}^{j-i}}{(\theta^*_i-\theta^*_j)(\theta^*_i - \theta^*_{j-1}) \cdots (\theta^*_i - \theta^*_{i+1})}.
\end{align*}
\end{proposition}
\begin{proof} By Lemma \ref{lem:EFstep1} and induction on $j-i$.
\end{proof}


\begin{theorem}
\label{prop:comp}
For $0 \leq i ,j\leq d$ we have
\begin{align*}
&F_i E^*_i A E^*_j \\
&= \sum_{i \leq s \leq j} \theta_s \frac{ \mathcal L^{s-i}}{ (\theta^*_i-\theta^*_s)(\theta^*_i-\theta^*_{s-1}) \cdots (\theta^*_i - \theta^*_{i+1})} \; 
\frac{\mathcal L^{j-s}}{(\theta^*_j - \theta^*_s) (\theta^*_j - \theta^*_{s+1})\cdots (\theta^*_j - \theta^*_{j-1})} F_j E^*_j \\
&+ \sum_{\stackrel{i \leq r \leq d,}{\stackrel{0\leq s \leq j, }{\stackrel{r-s=1}{}}}}
 \frac{ \mathcal L^{r-i}}{ (\theta^*_i-\theta^*_r)(\theta^*_i-\theta^*_{r-1})\cdots (\theta^*_i - \theta^*_{i+1})} \; \mathcal R \; \frac{\mathcal L^{j-s}}{(\theta^*_j - \theta^*_s)(\theta^*_j - \theta^*_{s+1})\cdots (\theta^*_j - \theta^*_{j-1})} F_j E^*_j.
\end{align*}
\end{theorem}
\begin{proof}
We have
\begin{align*}
F_i E^*_i A E^*_j &= F_i E^*_i \Biggl( \sum_{r=0}^d F_r \Biggr) A \Biggl( \sum_{s=0}^d F_s \Biggr)E^*_j \\
&=\sum_{r=0}^d \sum_{s=0}^d F_i E^*_i F_r A F_s E^*_j .
\end{align*}
In the previous double sum, for $0 \leq r,s\leq d$ the $(r,s)$-summand is zero unless
$i\leq r$ and $s \leq j$ and $r-s\in \lbrace 0,1\rbrace$. Therefore,
\begin{align}
F_i E^*_i A E^*_j &= \sum_{i \leq s \leq j} F_i E^*_i F_s A F_s E^*_j 
                            + \sum_{\stackrel{i \leq r \leq d,}{\stackrel{0\leq s \leq j, }{\stackrel{r-s=1}{}}}} F_i E^*_i F_r A F_s E^*_j.           \label{eq:sumsum}
\end{align}
We evaluate the terms in \eqref{eq:sumsum} using Lemmas \ref{lem:RRmeaning2}, \ref{lem:LLmeaning2}  and Propositions \ref{prop:FEform}, \ref{prop:EFform}.
For $i \leq s \leq j$ we have
\begin{align*}
&F_i E^*_i F_s A F_s E^*_j 
= \theta_s F_i E^*_i F_s E^*_j \\
&= \theta_s F_i E^*_i F_i \frac{ \mathcal L^{s-i}}{(\theta^*_i - \theta^*_s)(\theta^*_i-\theta^*_{s-1})\cdots (\theta^*_i - \theta^*_{i+1})} E^*_j \\
&= \theta_s F_i \frac{ \mathcal L^{s-i}}{(\theta^*_i - \theta^*_s)(\theta^*_i-\theta^*_{s-1})\cdots (\theta^*_i - \theta^*_{i+1})} E^*_j \\
&= \theta_s \frac{ \mathcal L^{s-i}}{(\theta^*_i - \theta^*_s)(\theta^*_i-\theta^*_{s-1})\cdots (\theta^*_i - \theta^*_{i+1})} F_s E^*_j \\
&= \theta_s \frac{ \mathcal L^{s-i}}{(\theta^*_i - \theta^*_s)(\theta^*_i-\theta^*_{s-1})\cdots (\theta^*_i - \theta^*_{i+1})}\; \frac{\mathcal L^{j-s}}{ (\theta^*_j - \theta^*_s)(\theta^*_j - \theta^*_{s+1})
\cdots (\theta^*_j - \theta^*_{j-1})}        F_j E^*_j. 
\end{align*}
For each ordered pair of integers $r,s$ such that $i \leq r \leq d$ and $0 \leq s \leq j$ and $r-s=1$, we have
\begin{align*}
&F_i E^*_i F_r A F_s E^*_j \\
&= F_i E^*_i F_i  \frac{ \mathcal L^{r-i}}{(\theta^*_i - \theta^*_r)(\theta^*_i-\theta^*_{r-1})\cdots (\theta^*_i - \theta^*_{i+1})} A F_s E^*_j \\
&= F_i  \frac{ \mathcal L^{r-i}}{(\theta^*_i - \theta^*_r)(\theta^*_i-\theta^*_{r-1})\cdots (\theta^*_i - \theta^*_{i+1})} A F_s E^*_j \\
&=  \frac{ \mathcal L^{r-i}}{(\theta^*_i - \theta^*_r)(\theta^*_i-\theta^*_{r-1})\cdots (\theta^*_i - \theta^*_{i+1})} F_r A F_s E^*_j \\
&=  \frac{ \mathcal L^{r-i}}{(\theta^*_i - \theta^*_r)(\theta^*_i-\theta^*_{r-1})\cdots (\theta^*_i - \theta^*_{i+1})} \mathcal R F_s E^*_j \\
&=  \frac{ \mathcal L^{r-i}}{(\theta^*_i - \theta^*_r)(\theta^*_i-\theta^*_{r-1})\cdots (\theta^*_i - \theta^*_{i+1})}\;\mathcal R   \;  \frac{\mathcal L^{j-s}}{ (\theta^*_j - \theta^*_s)(\theta^*_j - \theta^*_{s+1})\cdots (\theta^*_j - \theta^*_{j-1})}               F_jE^*_j.
\end{align*}
Evaluating \eqref{eq:sumsum} using these comments, we get the result.
\end{proof}


\noindent Our next general goal is to consider Theorem \ref{prop:comp} under the assumption $\vert j-i\vert \leq 1$.
There are three cases $j-i\in \lbrace -1,0,1\rbrace$. These cases are treated in the next three corollaries.

\begin{corollary} 
\label{cor:c1}
For $0 \leq j \leq d-1$,
\begin{align*}
F_{j+1} E^*_{j+1} A E^*_j = \mathcal R F_j E^*_j.
\end{align*}
\end{corollary}
\begin{proof} By Theorem \ref{prop:comp} with $i=j+1$.
\end{proof}

\begin{corollary}
\label{cor:c2}
For $1 \leq j \leq d-1$,
\begin{align*}
F_j E^*_j A E^*_j = \Biggl( \theta_j I +\frac{\mathcal R \mathcal L}{\theta^*_j-\theta^*_{j-1}} + \frac{ \mathcal L \mathcal R}{\theta^*_j-\theta^*_{j+1}}      \Biggr) F_j E^*_j.
\end{align*}
Moreover
\begin{align*}
F_0 E^*_0 A E^*_0 &= \Biggl( \theta_0 I + \frac{ \mathcal L \mathcal R}{\theta^*_0-\theta^*_{1}}      \Biggr) F_0 E^*_0, \\
F_d E^*_d A E^*_d &= \Biggl( \theta_d I +\frac{\mathcal R \mathcal L}{\theta^*_d-\theta^*_{d-1}}      \Biggr) F_d E^*_d.
\end{align*}
\end{corollary}
\begin{proof} By Theorem \ref{prop:comp} with $i=j$.
\end{proof}

\begin{corollary}
\label{cor:c3}
 For $2 \leq j \leq d-1$,
\begin{align*}
&F_{j-1} E^*_{j-1} A E^*_j \\
&=\biggl(
\frac{\theta_j - \theta_{j-1}}{\theta^*_{j-1}-\theta^*_j} \mathcal L + \frac{\mathcal R \mathcal L^2}{(\theta^*_j - \theta^*_{j-1})(\theta^*_j-\theta^*_{j-2})} 
         - \frac{\mathcal L \mathcal R \mathcal L}{(\theta^*_{j-1}-\theta^*_j)^2}+ \frac{{\mathcal L}^2 \mathcal R}{(\theta^*_{j-1}-\theta^*_j)(\theta^*_{j-1}-\theta^*_{j+1})} 
         \Biggr) F_j E^*_j.
\end{align*}
For $d\geq 2$,
\begin{align*}
F_{0} E^*_{0} A E^*_1 
&=\biggl(
\frac{\theta_1 - \theta_{0}}{\theta^*_{0}-\theta^*_1} \mathcal L 
         - \frac{\mathcal L \mathcal R \mathcal L}{(\theta^*_{0}-\theta^*_1)^2}+ \frac{{\mathcal L}^2 \mathcal R}{(\theta^*_{0}-\theta^*_1)(\theta^*_{0}-\theta^*_{2})} 
         \Biggr) F_1 E^*_1, \\
F_{d-1} E^*_{d-1} A E^*_d 
&=\biggl(
\frac{\theta_d - \theta_{d-1}}{\theta^*_{d-1}-\theta^*_d} \mathcal L + \frac{\mathcal R \mathcal L^2}{(\theta^*_d - \theta^*_{d-1})(\theta^*_d-\theta^*_{d-2})} 
         - \frac{\mathcal L \mathcal R \mathcal L}{(\theta^*_{d-1}-\theta^*_d)^2}
         \Biggr) F_d E^*_d.
\end{align*}
For $d=1$,
\begin{align*}
F_{0} E^*_{0} A E^*_1 
&=\biggl(
\frac{\theta_1 - \theta_{0}}{\theta^*_{0}-\theta^*_1} \mathcal L 
         - \frac{\mathcal L \mathcal R \mathcal L}{(\theta^*_{0}-\theta^*_1)^2}
         \Biggr) F_1 E^*_1.
\end{align*}

\end{corollary}
\begin{proof} By Theorem \ref{prop:comp} with $i=j-1$.
\end{proof} 

\noindent In the next three results, we interpret Corollaries \ref{cor:c1}--\ref{cor:c3} using commuting diagrams.

\begin{theorem}
\label{thm:m1} For $0 \leq j \leq d-1$
the following diagram commutes:
\begin{align*}
{\begin{CD}
E^*_jV @> R  >>
               E^*_{j+1}V
              \\
         @V \sum_{\ell=0}^d F_\ell E^*_\ell  VV                   @VV \sum_{\ell=0}^d F_\ell E^*_\ell V \\
         F_jV@>>\mathcal R  >
                                 F_{j+1}V
                        \end{CD}}
\end{align*}
\end{theorem}
\begin{proof} This is a reformulation of Corollary \ref{cor:c1}.
\end{proof}

\begin{theorem}
\label{thm:m2} For $1 \leq j \leq d-1$
the following diagram commutes:
\begin{align*}
{\begin{CD}
E^*_jV @> F  >>
               E^*_{j}V
              \\
         @V \sum_{\ell=0}^d F_\ell E^*_\ell  VV                   @VV \sum_{\ell=0}^d F_\ell E^*_\ell V \\
         F_jV@>>\theta_j I + \frac{\mathcal R \mathcal L}{\theta^*_j- \theta^*_{j-1}} + \frac{\mathcal L \mathcal R}{\theta^*_j-\theta^*_{j+1}} >
                                 F_{j}V
                        \end{CD}}
\end{align*}
Moreover, the following diagrams commute:
\begin{align*}
{\begin{CD}
E^*_0V @> F  >>
               E^*_{0}V
              \\
         @V \sum_{\ell=0}^d F_\ell E^*_\ell  VV                   @VV \sum_{\ell=0}^d F_\ell E^*_\ell V \\
         F_0V@>>\theta_0 I + \frac{\mathcal L \mathcal R}{\theta^*_0-\theta^*_{1}} >
                                 F_{0}V
                        \end{CD}}
\end{align*}
\begin{align*}
{\begin{CD}
E^*_dV @> F  >>
               E^*_{d}V
              \\
         @V \sum_{\ell=0}^d F_\ell E^*_\ell  VV                   @VV \sum_{\ell=0}^d F_\ell E^*_\ell V \\
         F_dV@>>\theta_d I + \frac{\mathcal R \mathcal L}{\theta^*_d- \theta^*_{d-1}}  >
                                 F_{d}V
                        \end{CD}}
\end{align*}

\end{theorem}
\begin{proof}
This is a reformulation of Corollary \ref{cor:c2}.
\end{proof}

\begin{theorem}
\label{thm:m3} For $2 \leq j\leq d-1$
the following diagram commutes:
\begin{align*}
{\begin{CD}
E^*_jV @> L  >>
               E^*_{j-1}V
              \\
         @V \sum_{\ell=0}^d F_\ell E^*_\ell  VV                   @VV \sum_{\ell=0}^d F_\ell E^*_\ell V \\
         F_j V@>>\frac{\theta_j - \theta_{j-1}}{\theta^*_{j-1}-\theta^*_j} \mathcal L + \frac{\mathcal R \mathcal L^2}{(\theta^*_j - \theta^*_{j-1})(\theta^*_j-\theta^*_{j-2})} 
         - \frac{\mathcal L \mathcal R \mathcal L}{(\theta^*_{j-1}-\theta^*_j)^2}+ \frac{{\mathcal L}^2 \mathcal R}{(\theta^*_{j-1}-\theta^*_j)(\theta^*_{j-1}-\theta^*_{j+1})} >
                                 F_{j-1}V
                        \end{CD}}
\end{align*}
For $d\geq 2$ the following diagrams commute:
\begin{align*}
{\begin{CD}
E^*_1V @> L  >>
               E^*_{0}V
              \\
         @V \sum_{\ell=0}^d F_\ell E^*_\ell  VV                   @VV \sum_{\ell=0}^d F_\ell E^*_\ell V \\
         F_1 V@>>\frac{\theta_1 - \theta_{0}}{\theta^*_{0}-\theta^*_1} \mathcal L 
         - \frac{\mathcal L \mathcal R \mathcal L}{(\theta^*_{0}-\theta^*_1)^2}+ \frac{{\mathcal L}^2 \mathcal R}{(\theta^*_{0}-\theta^*_1)(\theta^*_{0}-\theta^*_{2})} >
                                 F_{0}V
                        \end{CD}}
\end{align*}
\begin{align*}
{\begin{CD}
E^*_dV @> L  >>
               E^*_{d-1}V
              \\
         @V \sum_{\ell=0}^d F_\ell E^*_\ell  VV                   @VV \sum_{\ell=0}^d F_\ell E^*_\ell V \\
         F_d V@>>\frac{\theta_d - \theta_{d-1}}{\theta^*_{d-1}-\theta^*_d} \mathcal L + \frac{\mathcal R \mathcal L^2}{(\theta^*_d - \theta^*_{d-1})(\theta^*_d-\theta^*_{d-2})} 
         - \frac{\mathcal L \mathcal R \mathcal L}{(\theta^*_{d-1}-\theta^*_d)^2} >
                                 F_{d-1}V
                        \end{CD}}
\end{align*}
For $d=1$ the following diagram commutes:
\begin{align*}
{\begin{CD}
E^*_1V @> L  >>
               E^*_{0}V
              \\
         @V \sum_{\ell=0}^d F_\ell E^*_\ell  VV                   @VV \sum_{\ell=0}^d F_\ell E^*_\ell V \\
         F_1 V@>>\frac{\theta_1 - \theta_{0}}{\theta^*_{0}-\theta^*_1} \mathcal L 
         - \frac{\mathcal L \mathcal R \mathcal L}{(\theta^*_{0}-\theta^*_1)^2} >
                                 F_{0}V
                        \end{CD}}
\end{align*}

\end{theorem}
\begin{proof}
This is a reformulation of Corollary \ref{cor:c3}.
\end{proof}

\section{How $\mathcal R, \mathcal L$ are related}
We continue to discuss the tridiagonal system $\Phi = ( A; \lbrace E_i \rbrace_{i=0}^d; A^*; \lbrace E^*_i \rbrace_{i=0}^d )$ on $V$.
In this section we obtain some relations involving $\mathcal R$, $\mathcal L$. To this end,
we consider Theorem  \ref{prop:comp} under the assumption $j-i\geq 2$.

\begin{theorem}
\label{thm:rest}
For $0 \leq i,j\leq d$ such that $j-i \geq 2$, the following holds on $F_jV$:
\begin{align*}
0&= \sum_{i \leq s \leq j} \theta_s \frac{ \mathcal L^{s-i}}{ (\theta^*_i-\theta^*_s) (\theta^*_i - \theta^*_{s-1})\cdots (\theta^*_i - \theta^*_{i+1})} \; \frac{\mathcal L^{j-s}}{(\theta^*_j - \theta^*_s)(\theta^*_j - \theta^*_{s+1})\cdots (\theta^*_j - \theta^*_{j-1})} \\
&+ \sum_{\stackrel{i \leq r\leq d, }{\stackrel{0\leq s \leq j, }{\stackrel{r-s=1}{}}}}
 \frac{ \mathcal L^{r-i}}{ (\theta^*_i-\theta^*_r) (\theta^*_i - \theta^*_{r-1})\cdots (\theta^*_i - \theta^*_{i+1})} \; \mathcal R \; \frac{\mathcal L^{j-s}}{(\theta^*_j - \theta^*_s)(\theta^*_j - \theta^*_{s+1})\cdots (\theta^*_j - \theta^*_{j-1})}.
\end{align*}
\end{theorem}
\begin{proof} Let $S_{i,j}$ denote the expression on the right-hand side of the previous equation. We show that $0=S_{i,j}$ on $F_jV$. It suffices to show that $0=S_{i,j}F_j$.
Note that $E^*_i A E^*_j =0$ because $j-i \geq 2$. By this and Theorem \ref{prop:comp}, we obtain $0 = S_{i,j} F_j E^*_j$. Therefore
$0 = S_{i,j} F_j E^*_j F_j = S_{i,j} F_j$ in view of Lemma \ref{lem:FEF}.
\end{proof}

\begin{theorem}
\label{thm:restDual}
For $0 \leq i,j\leq d$ such that $j-i \geq 2$, the following holds on $F_iV$:
\begin{align*}
0&= \sum_{i \leq s \leq j} \theta^*_s 
\frac{\mathcal R^{j-s}}{(\theta_j - \theta_s)(\theta_j - \theta_{s+1}) \cdots (\theta_j - \theta_{j-1})} \;
\frac{ \mathcal R^{s-i}}{ (\theta_i-\theta_s) (\theta_i - \theta_{s-1})\cdots (\theta_i - \theta_{i+1})}
 \\
&+ \sum_{\stackrel{i \leq r\leq d, }{\stackrel{0\leq s \leq j, }{\stackrel{r-s=1}{}}}}
 \frac{\mathcal R^{j-s}}{(\theta_j - \theta_s)(\theta_j - \theta_{s+1})\cdots (\theta_j - \theta_{j-1})}
 \; \mathcal L \; 
 \frac{ \mathcal R^{r-i}}{ (\theta_i-\theta_r)(\theta_i - \theta_{r-1}) \cdots (\theta_i - \theta_{i+1})}.
\end{align*}
\end{theorem}
\begin{proof} Apply Theorem \ref{thm:rest} to the tridiagonal system
 $\Phi^\times = (A^*; \lbrace E^*_{d-\ell} \rbrace_{\ell=0}^d; A; \lbrace E_{d-\ell} \rbrace_{\ell=0}^d)$.
\end{proof}
\noindent  In \cite[Theorem~12.1]{TD00},  the tridiagonal relations are used to derive some relations involving $\mathcal R, \mathcal L$.
In the next result,  we show that the same relations involving $\mathcal R, \mathcal L$ follow from Theorems \ref{thm:rest}, \ref{thm:restDual} with $j-i=2$.

\begin{corollary}
\label{lem:RRRL}
{\rm (See \cite[Theorem~12.1]{TD00}.)}
For $2 \leq j \leq d$ the map
\begin{align} \label{eq:omega1}
\mathcal R \mathcal L^3 - (\beta+1)\mathcal L\mathcal R\mathcal L^2+(\beta+1)\mathcal L^2\mathcal R\mathcal L-\mathcal L^3\mathcal R - (\beta+1)e_j \mathcal L^2
\end{align}
vanishes on $F_{j}V$
and the map
\begin{align} \label{eq:omega2}
\mathcal R^3 \mathcal L - (\beta+1)\mathcal R^2 \mathcal L \mathcal R+(\beta+1) \mathcal R\mathcal L \mathcal R^2- \mathcal L\mathcal R^3 - (\beta+1)e_j \mathcal R^2
\end{align}
vanishes on $F_{j-2}V$, where 
\begin{align}
e_j = (\theta_{j-1}-\theta_{j-2})(\theta^*_{j-1} - \theta^*_{j-2})-(\theta_{j-1} - \theta_j)(\theta^*_{j-1} - \theta^*_j).                \label{eq:Omega}
\end{align}
\end{corollary}
\begin{proof} We verify the first assertion. 
Let $\Omega_j$ denote the expression in \eqref{eq:omega1}.
 We show that $\Omega_j=0$ on $F_jV$. 
  For notational convenience, define
 \begin{align*}
  L_j = \begin{cases} \frac{\mathcal R \mathcal L^3} { (\theta^*_j-\theta^*_{j-3})(\theta^*_j - \theta^*_{j-2})(\theta^*_j - \theta^*_{j-1})} & {\rm if} \;\; j\geq 3, \\
        0 & {\rm if} \;\;j=2
 \end{cases}
 \qquad \qquad  L^\vee_j = \begin{cases} 0 & {\rm if} \;\; j\geq 3, \\
        \mathcal R \mathcal L^3  & {\rm if} \;\;j=2
 \end{cases}
 \end{align*}
  \begin{align*}
 H_j = \begin{cases} \frac{\mathcal L^3\mathcal R}{(\theta^*_{j-2} - \theta^*_{j+1})(\theta^*_{j-2} - \theta^*_j)(\theta^*_{j-2} - \theta^*_{j-1})} & {\rm if} \;\; j\leq d-1, \\
        0 & {\rm if} \;\;j=d
 \end{cases}
 \qquad   H^\vee_j = \begin{cases} 0 & {\rm if} \;\; j\leq d-1, \\
       \mathcal L^3 \mathcal R & {\rm if} \;\;j=d.
 \end{cases}
  \end{align*}
 Setting $i=j-2$ in Theorem \ref{thm:rest} we find that on $F_jV$,
\begin{align*}
0&=\theta_{j-2} \frac{\mathcal L^2}{ (\theta^*_j-\theta^*_{j-2})(\theta^*_j-\theta^*_{j-1})}
+
\theta_{j-1} \frac{\mathcal L}{\theta^*_{j-2}-\theta^*_{j-1}} \, \frac{\mathcal L}{\theta^*_j - \theta^*_{j-1}}
\\
&+
\theta_j \frac{\mathcal L^2}{(\theta^*_{j-2}-\theta^*_{j})(\theta^*_{j-2}-\theta^*_{j-1})}
+
L_j 
+
\frac{\mathcal L}{\theta^*_{j-2}-\theta^*_{j-1}} \,\mathcal R \,\frac{\mathcal L^2}{(\theta^*_j - \theta^*_{j-2})(\theta^*_j - \theta^*_{j-1})}
\\
&+
\frac{\mathcal L^2}{(\theta^*_{j-2} - \theta^*_j)(\theta^*_{j-2} - \theta^*_{j-1})}\,\mathcal R \,\frac{\mathcal L}{ \theta^*_j - \theta^*_{j-1}}
+ H_j. 
\end{align*}
Let $S_j$ denote the expression on the right-hand side of the previous equation.
We claim that
\begin{align*}
\Omega_j = S_j  (\beta+1) ( \theta^*_{j-2}-\theta^*_j)(\theta^*_{j-2}-\theta^*_{j-1}) (\theta^*_j - \theta^*_{j-1}) + L^\vee_j - H^\vee_j.
\end{align*}
To prove the claim, we break the argument into four cases
 \begin{align*}
  3 \leq j \leq d-1, \qquad 2=j \leq d-1, \qquad 3 \leq j=d, \qquad 2=j=d.
  \end{align*}
In each case, the claim is proved by direct calculation using \eqref{eq:beta}.
We now use the claim.
 On $F_jV$ we have  $S_j=0$ and $L^\vee_j=0$ and $H^\vee_j=0$. By these comments and the claim, we obtain $\Omega_j=0$ on $F_jV$.
 We have verified the first assertion. The second assertion is similarly obtained using Theorem \ref{thm:restDual}.
\end{proof}

\section{More bijections}
We continue to discuss the tridiagonal system $\Phi = ( A; \lbrace E_i \rbrace_{i=0}^d; A^*; \lbrace E^*_i \rbrace_{i=0}^d )$ on $V$.
In Lemma \ref{lem:RRbij}, we described some bijections involving $\mathcal R, \mathcal L$.
In this section, we obtain some analogous bijections involving $R, L$.

\begin{theorem}
\label{thm:bij2}
For $0 \leq i\leq j\leq d$ the following {\rm (i), (ii)} hold.
\begin{enumerate}
\item[\rm (i)] The map ${R}^{j-i}: E^*_iV \to E^*_jV$ is injective if $i+j\leq d$, bijective if $i+j=d$, and surjective if $i+j\geq d$.
\item[\rm (ii)]  The map ${L}^{j-i}: E^*_jV \to E^*_iV$ is injective if $i+j\geq d$, bijective if $i+j=d$, and surjective if $i+j\leq d$.
\end{enumerate}
\end{theorem}
\begin{proof} (i) By Lemma \ref{lem:RRbij}(i) and Theorem \ref{thm:m1}. \\
\noindent (ii) Apply (i) above to the tridiagonal system $\Phi^\downarrow = (A; \lbrace E_{\ell} \rbrace_{\ell=0}^d; A^*; \lbrace E^*_{d-\ell} \rbrace_{\ell=0}^d)$.
\end{proof}

\begin{corollary}
\label{cor:bij3}
For $0 \leq i\leq j\leq d$,
\begin{align}
{\rm rank} \,E^*_i A^{j-i} E^*_j = {\rm rank} \,E^*_j A^{j-i} E^*_i = \hbox{\rm min} \lbrace \rho_i, \rho_j \rbrace = \begin{cases} \rho_i & {\rm if}\; i+j\leq d; \\
                                                                                                                                                               \rho_j  &{\rm if}\;  i + j \geq d
                                                                                                                                                               \end{cases}
                                                                                                                                                               \label{eq:rank1}
                                                                                                                                                               \end{align}
 and also
 \begin{align}
{\rm rank} \,E_i (A^*)^{j-i} E_j = {\rm rank} \,E_j (A^*)^{j-i} E_i = \hbox{\rm min} \lbrace \rho_i, \rho_j \rbrace = \begin{cases} \rho_i & {\rm if}\; i+j\leq d; \\
                                                                                                                                                               \rho_j  &{\rm if}\;  i + j \geq d.
                                                                                                                                                               \end{cases}
                                                                                                                                                               \label{eq:rank2}
                                                                                                                                                               \end{align}
                                                                                                                                                               \end{corollary}
\begin{proof}
Assertion \eqref{eq:rank1} follows from \eqref{eq:unimodal} and Theorem \ref{thm:bij2}. To prove  \eqref{eq:rank2}, apply
\eqref{eq:rank1} to the tridiagonal system $\Phi^* = (A^*; \lbrace E^*_\ell \rbrace_{\ell=0}^d; A; \lbrace E_\ell \rbrace_{\ell=0}^d)$.
\end{proof}

\section{Assuming $\Phi$ is a Leonard system}
We continue to discuss the tridiagonal system $\Phi = ( A; \lbrace E_i \rbrace_{i=0}^d; A^*; \lbrace E^*_i \rbrace_{i=0}^d )$ on $V$.
Recall the shape $\lbrace \rho_i \rbrace_{i=0}^d$ of $\Phi$.
We call $\Phi$ a {\it Leonard system} whenever $\rho_i = 1$ for $0 \leq i \leq d$. In this case ${\rm dim}\, V=d+1$. See 
\cite{smLP, 2lintrans, ter24, terCanForm,ter2004, ter2005, ter2005b, madrid, LSnotes, vidunas} for some basic facts about Leonard systems.
Throughout this section, we assume that $\Phi$ is a Leonard system. Our goal is to describe what our earlier results become
 under this assumption.
 
 \begin{lemma} 
 \label{lem:bij2}
 For $0 \leq i\leq j \leq d$ the following {\rm (i)--(iv)} hold:
 \begin{enumerate}
 \item[\rm (i)]  the map $\mathcal R^{j-i}: F_iV \to F_j V$ is a bijection;
 \item[\rm (ii)] the map $\mathcal L^{j-i} : F_j V \to F_iV$ is a bijection;
 \item[\rm (iii)]  the map $ R^{j-i}: E^*_iV \to E^*_j V$ is a bijection;
 \item[\rm (iv)]  the map $ L^{j-i}: E^*_jV \to E^*_i V$ is a bijection.
 \end{enumerate}
 \end{lemma} 
 \begin{proof} (i), (ii) By Lemma \ref{lem:RRbij} and since $F_\ell V$ has dimension 1 for $0 \leq \ell \leq d$. \\
 (iii), (iv) By Theorem \ref{thm:bij2} and since $E^*_\ell V$ has dimension 1 for $0 \leq \ell \leq d$.
 \end{proof}
\noindent We bring in some notation. Let ${\rm Mat}_{d+1}(\mathbb F)$ denote the algebra 
consisting of the $d+1$ by $d+1$ matrices that have all entries in $\mathbb F$. We index the rows and columns by $0,1,\ldots, d$.
Let $\lbrace v_i \rbrace_{i=0}^d$ denote a basis for $V$. For $X \in {\rm Mat}_{d+1}(\mathbb F)$ and $Y \in {\rm End}(V)$, we say that
{\it $X$ represents $Y$ with respect to $\lbrace v_i \rbrace_{i=0}^d$} whenever $Y v_j = \sum_{i=0}^d X_{i,j} v_i$ for
$0 \leq j \leq d$.
\medskip

\noindent We recall from \cite[Section~7]{ter2004}  the scalars $\lbrace a_i \rbrace_{i=0}^d$, $\lbrace x_i \rbrace_{i=1}^d$. Pick $0 \not=\zeta \in E^*_0V$.
By Lemma \ref{lem:bij2}(iii), for $0 \leq i \leq d$ the vector $R^i \zeta$ is a basis for $E^*_iV$. Thus the vectors $\lbrace R^i \zeta\rbrace_{i=0}^d$
form a basis for $V$. With respect to this basis we have the following matrix representations in ${\rm Mat}_{d+1}(\mathbb F)$.
For $0 \leq i \leq d$ we have
\begin{align*}
E^*_i:\quad \hbox{\rm diag} ( 0, \ldots, 0,1,0,\ldots, 0)
\end{align*}
where the $1$ is in the $(i,i)$-entry. Using Lemma \ref{lem:EAEp}(i) and  \eqref{eq:AAs} 
 we obtain
\begin{align*}
A: \quad \begin{pmatrix}
a_0 & x_1 & &&& {\bf 0} \\
1 &a_1 &x_2 &&& \\
&1&\cdot &\cdot &&\\
&&\cdot &\cdot &\cdot &\\
&&&\cdot &\cdot & x_d\\
{\bf 0}&&&&1 &a_d 
\end{pmatrix},
\qquad \qquad
A^*:\quad \hbox{\rm diag} ( \theta^*_0, \theta^*_1, \ldots, \theta^*_d)
\end{align*}
where $x_i \not=0$  $(1 \leq i \leq d)$. By Definition \ref{def:RFL}  we have
\begin{align*}
&R: \quad \begin{pmatrix}
0 &  & &&& {\bf 0} \\
1 &0 & &&& \\
&1&\cdot &&&\\
&&\cdot &\cdot & &\\
&&&\cdot &\cdot & \\
{\bf 0}&&&&1 &0 
\end{pmatrix},
\qquad \qquad
L: \quad \begin{pmatrix}
0 & x_1 & &&& {\bf 0} \\
 &0 &x_2 &&& \\
&&\cdot &\cdot &&\\
&& &\cdot &\cdot &\\
&&&&\cdot & x_d\\
{\bf 0}&&&&&0
\end{pmatrix},
\end{align*}
\begin{align*}
F: \quad {\rm diag} (a_0, a_1, \ldots, a_d).
\end{align*}

\noindent
Using the above matrix representations (or \cite[Lemma~7.5]{ter2004}) we obtain
\begin{align}
E^*_i A E^*_i &= a_i E^*_i \qquad \qquad (0 \leq i \leq d),     
  \label{eq:EAEai}     \\
  E^*_i A E^*_{i-1} A E^*_i &= x_i  E^*_i \qquad \qquad (1 \leq i \leq d),
  \label{eq:EAEAEcibi}\\
   E^*_i A E^*_{i+1} A E^*_i &= x_{i+1} E^*_i \qquad \qquad (0 \leq i \leq d-1).
  \label{eq:EAEAEbici}
  \end{align}
  
\noindent  Next, we describe Theorem \ref{thm:ggee} under the assumption that $\Phi$ is a Leonard system.

\begin{theorem} 
\label{thm:TDaixi}
The following {\rm (i), (ii)} hold:
\begin{enumerate}
\item[\rm (i)]  For $2 \leq i \leq d$ we have
\begin{align*}
g^-_i a_{i-2} + a_{i-1} + g^+_i a_i = \gamma,
\end{align*}
where $g^+_i$, $g^-_i$ are from Theorem \ref{thm:ggee}(i).
\item[\rm (ii)] For $1 \leq i \leq d$ we have
\begin{align*}
e^-_i x_{i-1} + (\beta+2) x_i + e^+_i x_{i+1}
+ a^2_i - \beta a_{i-1} a_i + a_{i-1}^2 = \gamma (a_i + a_{i-1}) + \varrho,
\end{align*}
where $e^+_i, e^-_i$ are from Theorem \ref{thm:ggee}(ii) and $x_0=0=x_{d+1}$.
\end{enumerate}
\end{theorem}
\begin{proof} Evaluate Theorem \ref{thm:ggee} using the matrix representations below Lemma \ref{lem:bij2}.
  \end{proof}

\noindent We recall from \cite[Definition~3.10]{2lintrans} the first split sequence of $\Phi$. Let $\lambda$ denote an indeterminate, and let $\mathbb F\lbrack \lambda \rbrack$
denote the algebra of polynomials in $\lambda$ that have all coefficients in $\mathbb F$.
For $0 \leq i \leq d$ define $\tau_i, \tau^*_i \in {\mathbb F}\lbrack \lambda \rbrack$ by
\begin{align*}
\tau_i  &= (\lambda - \theta_0) (\lambda-\theta_1) \cdots (\lambda-\theta_{i-1}), \\
\tau^*_i &= (\lambda - \theta^*_0) (\lambda-\theta^*_1) \cdots (\lambda-\theta^*_{i-1}).
\end{align*}
The polynomials $\tau_i, \tau^*_i$ have degree $i$ for $0 \leq i \leq d$. Recall the vector $0 \not= \zeta \in E^*_0V$. By \eqref{eq:usum1} and the construction,
$E^*_0V=U_0=F_0V$. For the moment, let $0 \leq i \leq d$. By Lemmas \ref{lem:RRact}, \ref{lem:RRmeaning} 
we have $\tau_i(A) \zeta = \mathcal R^i \zeta$, and by Lemma \ref{lem:bij2}(i)
the vector $\mathcal R^i \zeta $ is a basis for $F_iV$. Thus the vectors $\lbrace \mathcal R^i \zeta \rbrace_{i=0}^d$ form a basis for $V$.
With respect to this basis, we have the following matrix representations in  ${\rm Mat}_{d+1} (\mathbb F)$.
For $0 \leq i \leq d$ we have
\begin{align*}
F_i : \quad {\rm diag}(0,\ldots, 0,1,0,\ldots, 0)
\end{align*}
where the $1$ is in the $(i,i)$-entry.
By the construction (or \cite[Theorem ~3.2]{2lintrans}) we have
\begin{align*}
A: \quad \begin{pmatrix}
\theta_0 &  & &&& {\bf 0} \\
1 &\theta_1 & &&& \\
&1&\theta_2 &&&\\
&&\cdot &\cdot &&\\
&&&\cdot &\cdot & \\
{\bf 0}&&&&1&\theta_d 
\end{pmatrix},
\qquad \qquad
A^*:\quad 
\begin{pmatrix}
\theta^*_0 & \varphi_1 & &&& {\bf 0} \\
 &\theta^*_1 &\varphi_2 &&& \\
&&\theta^*_2  &\cdot &&\\
&&&\cdot &\cdot &\\
&&&&\cdot & \varphi_d\\
{\bf 0}&&&& &\theta^*_d 
\end{pmatrix},
\end{align*}
where $\varphi_i \not=0$ $(1 \leq i \leq d)$.
The sequence $\lbrace \varphi_i \rbrace_{i=1}^d$ is called the {\it first split sequence} of $\Phi$.
By Definition \ref{def:RRLL} we have
\begin{align*}
\mathcal R: \quad \begin{pmatrix}
0 &  & &&& {\bf 0} \\
1 &0 & &&& \\
&1&0 &&&\\
&&\cdot &\cdot &&\\
&&&\cdot &\cdot & \\
{\bf 0}&&&&1&0 
\end{pmatrix},
\qquad \qquad
\mathcal L:\quad 
\begin{pmatrix}
0 & \varphi_1 & &&& {\bf 0} \\
 &0 &\varphi_2 &&& \\
&&0  &\cdot &&\\
&&&\cdot &\cdot &\\
&&&&\cdot & \varphi_d\\
{\bf 0}&&&& &0 
\end{pmatrix}.
\end{align*}
Using the above matrix representations, we obtain
\begin{align}
&F_i \mathcal R \mathcal L = \mathcal R F_{i-1} \mathcal L = \mathcal R \mathcal L F_i = \varphi_i F_i \qquad \qquad (1 \leq i \leq d),  \label{eq:FRL} \\
& F_i \mathcal L \mathcal R = \mathcal L F_{i+1} \mathcal R = \mathcal L \mathcal R F_i = \varphi_{i+1} F_i \qquad \qquad (0 \leq i \leq d-1). \label{eq:FLR}
\end{align}

\noindent The next result is well known; our goal is to derive it from Corollary \ref{cor:c2}.
\begin{proposition}
\label{prop:3termai}
{\rm (See \cite[Lemma~5.1]{2lintrans}.)}
For $1 \leq i \leq d-1$, 
\begin{align}
a_i &= \theta_i + \frac{\varphi_i}{\theta^*_i - \theta^*_{i-1}} + \frac{\varphi_{i+1}}{\theta^*_i - \theta^*_{i+1}}.
\label{eq:aiSUM}
\end{align}
Moreover,
\begin{align}
a_0 = \theta_0+ \frac{\varphi_{1}}{\theta^*_0 - \theta^*_{1}}, \qquad \qquad 
a_d = \theta_d + \frac{\varphi_d}{\theta^*_d - \theta^*_{d-1}}.
\label{eq:a0d}
\end{align}
\end{proposition}
\begin{proof} We first verify \eqref{eq:aiSUM}.
Using Corollary \ref{cor:c2} and  \eqref{eq:EAEai}, \eqref{eq:FRL}, \eqref{eq:FLR} we obtain
\begin{align*}
a_i F_i E^*_i &= F_i E^*_i A E^*_i \\
&= \biggl(  \theta_i I + \frac{ \mathcal R \mathcal L}{ \theta^*_i - \theta^*_{i-1}} + \frac{\mathcal L \mathcal R}{\theta^*_i - \theta^*_{i+1}}           \biggr) F_i E^*_i \\
&= \biggl(  \theta_i  + \frac{\varphi_i }{ \theta^*_i - \theta^*_{i-1}} + \frac{\varphi_{i+1} }{\theta^*_i - \theta^*_{i+1}}           \biggr) F_i E^*_i .
\end{align*}
By these comments and $F_i E^*_i \not=0$, we obtain \eqref{eq:aiSUM}. The assertions \eqref{eq:a0d} are similarly obtained.
\end{proof}

\noindent
Next, we describe Corollary \ref{cor:c3} under the assumption that $\Phi$ is a Leonard system.
\begin{proposition}
\label{prop:m4}
For $2 \leq i \leq d-1$,
\begin{align} \label{eq:first}
&\frac{(\theta^*_{i-1} - \theta^*_i)^2 x_i}{\varphi_i}
 =\varphi_{i-1} \frac{\theta^*_i - \theta^*_{i-1}}{\theta^*_i - \theta^*_{i-2}} - \varphi_i + \varphi_{i+1} \frac{\theta^*_{i-1} - \theta^*_i}{\theta^*_{i-1} - \theta^*_{i+1}}
- (\theta_{i-1}-\theta_i)(\theta^*_{i-1}-\theta^*_i).
\end{align}
For $d\geq 2$,
\begin{align} \label{eq:second}
&\frac{(\theta^*_{0} - \theta^*_1)^2 x_1}{\varphi_1}
 =      - \varphi_1+ \varphi_{2} \frac{\theta^*_{0} - \theta^*_1}{\theta^*_{0} - \theta^*_{2}}
- (\theta_{0}-\theta_1)(\theta^*_{0}-\theta^*_1), \\
&\frac{(\theta^*_{d-1} - \theta^*_d)^2 x_d}{\varphi_d}
 =\varphi_{d-1} \frac{\theta^*_d - \theta^*_{d-1}}{\theta^*_d - \theta^*_{d-2}} - \varphi_d 
- (\theta_{d-1}-\theta_d)(\theta^*_{d-1}-\theta^*_d). \label{eq:third}
\end{align}
For $d=1$,
\begin{align} \label{eq:forth}
\frac{(\theta^*_{0} - \theta^*_1)^2 x_1}{\varphi_1}
 =      - \varphi_1
- (\theta_{0}-\theta_1)(\theta^*_{0}-\theta^*_1).
\end{align}

\end{proposition}
\begin{proof} We first verify \eqref{eq:first}.
Using in order   \eqref{eq:EAEAEcibi}, Corollary \ref{cor:c1}, Corollary \ref{cor:c3}, Lemma \ref{lem:LLmeaning2}, \eqref{eq:FRL}, \eqref{eq:FLR}
we obtain
\begin{align*}
&(\theta^*_{i-1}-\theta^*_{i})^2 x_i F_i E^*_i \\
&\quad = (\theta^*_{i-1}-\theta^*_{i})^2 F_i E^*_i A E^*_{i-1} A E^*_i \\
&\quad = (\theta^*_{i-1}-\theta^*_{i})^2 \mathcal R F_{i-1}E^*_{i-1} A E^*_i \\
&\quad= \mathcal R \biggl(
(\theta_i - \theta_{i-1})(\theta^*_{i-1}- \theta^*_{i}) \mathcal L 
+
\mathcal R \mathcal L^2  \frac{\theta^*_i - \theta^*_{i-1}}{\theta^*_i - \theta^*_{i-2}} -\mathcal L \mathcal R \mathcal L+ \mathcal L^2 \mathcal R \frac{\theta^*_{i-1} - \theta^*_i}{\theta^*_{i-1} - \theta^*_{i+1}}
\biggr) F_i E^*_i \\
&\quad= \mathcal R \biggl(
\mathcal R \mathcal L^2  \frac{\theta^*_i - \theta^*_{i-1}}{\theta^*_i - \theta^*_{i-2}} -\mathcal L \mathcal R \mathcal L+ \mathcal L^2 \mathcal R \frac{\theta^*_{i-1} - \theta^*_i}{\theta^*_{i-1} - \theta^*_{i+1}}
-
(\theta_{i-1} - \theta_{i})(\theta^*_{i-1}- \theta^*_{i}) \mathcal L 
\biggr) F_i E^*_i \\
&\quad= \varphi_i  \biggl(
\varphi_{i-1}  \frac{\theta^*_i - \theta^*_{i-1}}{\theta^*_i - \theta^*_{i-2}} -\varphi_i+  \varphi_{i+1}  \frac{\theta^*_{i-1} - \theta^*_i}{\theta^*_{i-1} - \theta^*_{i+1}}
-(\theta_{i-1} - \theta_{i})(\theta^*_{i-1}- \theta^*_{i}) 
\biggr) F_i E^*_i.
\end{align*}
Assertion  \eqref{eq:first} follows from these comments along with $F_i E^*_i \not=0$ and $\varphi_i \not=0$. The assertions \eqref{eq:second}--\eqref{eq:forth}
are similarly obtained.
\end{proof}

\noindent We recall from \cite[Section~11]{ter2004}  the scalars $\lbrace b_i \rbrace_{i=0}^{d-1}$, $\lbrace c_i \rbrace_{i=1}^d$. Pick $0 \not=\xi \in E_0V$. By \cite[Lemma~10.2]{ter2004}, for $0 \leq i \leq d$ the vector
$E^*_i \xi $ is nonzero and hence a basis for $E^*_iV$. Thus the vectors $\lbrace E^*_i \xi\rbrace_{i=0}^d$ form a basis for $V$.
With respect to this basis, we have the following matrix representations in ${\rm Mat}_{d+1}(\mathbb F)$.
 For $0 \leq i \leq d$ we have
\begin{align*}
E^*_i:\quad \hbox{\rm diag} ( 0, \ldots, 0,1,0,\ldots, 0)
\end{align*}
where the $1$ is in the $(i,i)$-entry.
By construction (or \cite[Lemmas~10.7, 11.3]{ter2004}) we have
\begin{align*}
A: \quad \begin{pmatrix}
a_0 & b_0 & &&& {\bf 0} \\
c_1 &a_1 &b_1 &&& \\
&c_2&\cdot &\cdot &&\\
&&\cdot &\cdot &\cdot &\\
&&&\cdot &\cdot & b_{d-1}\\
{\bf 0}&&&&c_d &a_d 
\end{pmatrix},
\qquad \qquad
A^*:\quad \hbox{\rm diag} ( \theta^*_0, \theta^*_1, \ldots, \theta^*_d)
\end{align*}
where $c_i \not=0$ $(1 \leq i \leq d)$ and $b_i \not=0$ $(0 \leq i \leq d-1)$. 
By Definition \ref{def:RFL}  we have
\begin{align*}
&R: \quad \begin{pmatrix}
0 &  & &&& {\bf 0} \\
c_1 &0 & &&& \\
&c_2&\cdot &&&\\
&&\cdot &\cdot & &\\
&&&\cdot &\cdot & \\
{\bf 0}&&&&c_d &0 
\end{pmatrix},
\qquad \qquad
L: \quad \begin{pmatrix}
0 & b_0 & &&& {\bf 0} \\
 &0 &b_1 &&& \\
&&\cdot &\cdot &&\\
&& &\cdot &\cdot &\\
&&&&\cdot & b_{d-1}\\
{\bf 0}&&&&&0
\end{pmatrix},
\end{align*}
\begin{align*}
F: \quad {\rm diag} (a_0, a_1, \ldots, a_d).
\end{align*}
\noindent
By the construction (or \cite[Lemma~11.2(ii)]{ter2004}) we have
\begin{align*}
c_i + a_i + b_i = \theta_0 \qquad \qquad (0 \leq i \leq d),
\end{align*}
where $c_0=0$ and $b_d=0$. 
\noindent By   \eqref{eq:EAEAEcibi} (or
\cite[Lemma~11.2(i)]{ter2004}) we have
\begin{align*}
x_i &= c_i b_{i-1} \qquad \qquad (1 \leq i \leq d).
\end{align*}
By this and
 \cite[Theorem~17.7(i)]{ter2004} we obtain
\begin{align*}
b_i &= \varphi_{i+1} \frac{\tau^*_i(\theta^*_i)}{\tau^*_{i+1}(\theta^*_{i+1})} \qquad \qquad (0 \leq i \leq d-1), \\
c_i &= \frac{x_i}{\varphi_i} \, \frac{\tau^*_i(\theta^*_i)}{\tau^*_{i-1}(\theta^*_{i-1})} \qquad \qquad (1 \leq i \leq d).
\end{align*}

\noindent The next result is a variation on Proposition \ref{prop:m4}.
\begin{theorem}
\label{thm:m4}
For $2 \leq i \leq d-1$,
\begin{align*}
&c_i (\theta^*_{i-1} - \theta^*_i)^2\frac{\tau^*_{i-1} (\theta^*_{i-1})}{\tau^*_i(\theta^*_i)}  \\
&\qquad =\varphi_{i-1} \frac{\theta^*_i - \theta^*_{i-1}}{\theta^*_i - \theta^*_{i-2}} - \varphi_i + \varphi_{i+1} \frac{\theta^*_{i-1} - \theta^*_i}{\theta^*_{i-1} - \theta^*_{i+1}}
- (\theta_{i-1}-\theta_i)(\theta^*_{i-1}-\theta^*_i).
\end{align*}
For $d\geq 2$,
\begin{align*}
&c_1 (\theta^*_{1} - \theta^*_0)  = - \varphi_1 + \varphi_{2} \frac{\theta^*_{0} - \theta^*_1}{\theta^*_{0} - \theta^*_{2}}
- (\theta_{0}-\theta_1)(\theta^*_{0}-\theta^*_1), \\
&c_d (\theta^*_{d-1} - \theta^*_d)^2\frac{\tau^*_{d-1} (\theta^*_{d-1})}{\tau^*_d(\theta^*_d)}  
 =\varphi_{d-1} \frac{\theta^*_d - \theta^*_{d-1}}{\theta^*_d - \theta^*_{d-2}} - \varphi_d - (\theta_{d-1}-\theta_d)(\theta^*_{d-1}-\theta^*_d).
\end{align*}
For $d=1$,
\begin{align*}
c_1 (\theta^*_{1} - \theta^*_0)  = - \varphi_1 
- (\theta_{0}-\theta_1)(\theta^*_{0}-\theta^*_1).
\end{align*}
\end{theorem}
\begin{proof} In Proposition \ref{prop:m4},
eliminate $x_i$ using $x_i=c_i b_{i-1}$ and
\begin{align*}
b_{i-1} = \varphi_i \frac{ \tau^*_{i-1} (\theta^*_{i-1})}{\tau^*_i(\theta^*_i)}.
\end{align*}
The result follows.
\end{proof}

\noindent Next, we describe Theorem \ref{thm:rest} under the assumption that $\Phi$ is a Leonard system.
\begin{theorem}
\label{thm:rest2}
For $0 \leq i,j\leq d$ such that $j-i \geq 2$, 
\begin{align*}
0&= \sum_{i \leq s \leq j} \theta_s \frac{ 1}{ (\theta^*_i-\theta^*_s)(\theta^*_i - \theta^*_{s-1})\cdots (\theta^*_i - \theta^*_{i+1})} \; \frac{1}{(\theta^*_j - \theta^*_s)(\theta^*_j - \theta^*_{s+1}) \cdots (\theta^*_j - \theta^*_{j-1})} \\
&+ \sum_{\stackrel{i \leq r\leq d, }{\stackrel{0\leq s \leq j, }{\stackrel{r-s=1}{}}}}
 \frac{1}{ (\theta^*_i-\theta^*_r)(\theta^*_i - \theta^*_{r-1}) \cdots (\theta^*_i - \theta^*_{i+1})} \; \varphi_r \; \frac{1}{(\theta^*_j - \theta^*_s)(\theta^*_j -\theta^*_{s+1})\cdots (\theta^*_j - \theta^*_{j-1})}.
\end{align*}
\end{theorem}
\begin{proof} Express Theorem \ref{thm:rest} using the matrix representations  below Proposition \ref{thm:TDaixi}, and simplify
the result using $\varphi_\ell \not=0$ $(1 \leq \ell \leq d)$.    
\end{proof}

\noindent Next, we describe Theorem \ref{thm:restDual} under the assumption that $\Phi$ is a Leonard system.
\begin{theorem}
\label{thm:rest3}
For $0 \leq i,j\leq d$ such that $j-i \geq 2$, 
\begin{align*}
0&= \sum_{i \leq s \leq j} \theta^*_s 
\frac{1}{(\theta_j - \theta_s)(\theta_j - \theta_{s+1})\cdots (\theta_j - \theta_{j-1})} \;
\frac{1}{ (\theta_i-\theta_s)(\theta_i - \theta_{s-1}) \cdots (\theta_i - \theta_{i+1})}
 \\
&+ \sum_{\stackrel{i \leq r\leq d, }{\stackrel{0\leq s \leq j, }{\stackrel{r-s=1}{}}}}
 \frac{1}{(\theta_j - \theta_s)(\theta_j - \theta_{s+1})\cdots (\theta_j - \theta_{j-1})}
 \; \varphi_r \; 
 \frac{1}{ (\theta_i-\theta_r)(\theta_i - \theta_{r-1}) \cdots (\theta_i - \theta_{i+1})}.
\end{align*}
\end{theorem}
\begin{proof} Express Theorem \ref{thm:restDual} using the matrix representations  below Proposition \ref{thm:TDaixi}, and simplify
the result using $\varphi_\ell \not=0$ $(1 \leq \ell \leq d)$.    
\end{proof}

\noindent Next, we describe Corollary \ref{lem:RRRL} under the assumption that $\Phi$ is a Leonard system.
\begin{theorem}
\label{thm:m5}
For $2 \leq j \leq d$,
\begin{align*}
&\varphi_{j-2} - (\beta+1)\varphi_{j-1} + (\beta+1) \varphi_j- \varphi_{j+1} \\
&\quad =
(\beta+1) \Bigl( (\theta_{j-1}-\theta_{j-2})(\theta^*_{j-1}-\theta^*_{j-2})- (\theta_{j-1} - \theta_j)(\theta^*_{j-1}-\theta^*_j)       \Bigr),
\end{align*}
where $\varphi_0=0$ and $\varphi_{d+1}=0$.
\end{theorem}
\begin{proof} Express Corollary \ref{lem:RRRL} using the matrix representations below Proposition \ref{thm:TDaixi},
and simplify
the result using $\varphi_\ell \not=0$ $(1 \leq \ell \leq d)$. 
\end{proof}

\section{Assuming $\Phi$ has Krawtchouk type}
We continue to discuss the tridiagonal system $\Phi = ( A; \lbrace E_i \rbrace_{i=0}^d; A^*; \lbrace E^*_i \rbrace_{i=0}^d )$ on $V$.
Recall the eigenvalue sequence $\lbrace \theta_i \rbrace_{i=0}^d$ and dual eigenvalue sequence $\lbrace \theta^*_i \rbrace_{i=0}^d$.
We say that $\Phi$ has {\it Krawtchouk type} whenever
\begin{align}
 \theta_i = d-2i, \qquad \qquad \theta^*_i = d-2i, \qquad \qquad (0 \leq i \leq d).
 \label{eq:Kraw}
  \end{align}
See \cite{jtgo, hartwig, ItoKraw, ItoTet, nomKraw}  for some basic facts about tridiagonal systems of Krawtchouk type. 
Throughout this section we assume that $\Phi$ has Krawtchouk type.
Our goal is to describe what our earlier results become under this assumption.
\medskip

\noindent By \eqref{eq:Kraw} and since $\lbrace \theta_i \rbrace_{i=0}^d$ are mutually distinct, the characteristic ${\rm Char}(\mathbb F)$ is either 0 or
an odd prime greater than $d$.
Evaluating Lemma \ref{lem:beta}  using \eqref{eq:Kraw}, we may assume
\begin{align}
\beta = 2, \qquad \gamma=0, \qquad \gamma^*=0, \qquad \varrho=4, \qquad \varrho^*=4. \label{eq:gamData}
\end{align}
Evaluating the tridiagonal relations \eqref{eq:TD1}, \eqref{eq:TD2} using \eqref{eq:gamData}, we obtain
\begin{align}
\lbrack A, \lbrack A, \lbrack A, A^* \rbrack \rbrack \rbrack &= 4 \lbrack A, A^* \rbrack, \label{eq:DG1} \\
\lbrack A^*, \lbrack A^*, \lbrack A^*, A \rbrack \rbrack \rbrack &= 4 \lbrack A^*, A \rbrack. \label{eq:DG2}
\end{align}
The relations \eqref{eq:DG1}, \eqref{eq:DG2} are called the Dolan/Grady relations; see for example \cite[Example~3.2]{qSerre}.
\medskip

\noindent 
Recall the maps $R,F,L$ from Definition \ref{def:RFL}.
\begin{lemma}
\label{lem:RFLKpre}
We have
\begin{align*}
\lbrack A^*, L \rbrack = 2 L, \qquad \qquad \lbrack A^*, F \rbrack = 0, \qquad \qquad \lbrack A^*, R \rbrack=-2R.
\end{align*}
\end{lemma}
\begin{proof}
We invoke Lemma \ref{lem:RFLmeaning}. For 
$0 \leq i \leq d$ the following holds on $E^*_iV$:
\begin{align*}
A^* L-L A^* - 2 L = \theta^*_{i-1} L - \theta^*_i L - 2L = (\theta^*_{i-1} - \theta^*_i -2 ) L = 0 L = 0.
\end{align*}
Therefore $\lbrack A^*, L \rbrack = 2 L$. The remaining assertions  are similarly obtained.
\end{proof}

\begin{lemma}
\label{lem:RFLK}
We have
\begin{align*}
R &= \frac{ \lbrack A^*, \lbrack A^*, A \rbrack \rbrack - 2 \lbrack A^*, A \rbrack}{8},  \\
F &= A - \frac{ \lbrack A^*, \lbrack A^*, A \rbrack \rbrack }{4},\\
L &= \frac{ \lbrack A^*, \lbrack A^*, A \rbrack \rbrack + 2 \lbrack A^*, A \rbrack}{8}.
\end{align*}
\end{lemma}
\begin{proof}  By Lemmas \ref{lem:ARFL} and \ref{lem:RFLKpre},
\begin{align*}
\lbrack A^*, A \rbrack = 2L-2R, \qquad \qquad  \lbrack A^*, \lbrack A^*, A \rbrack \rbrack= 4L+4R.
\end{align*}
Combining these equations with Lemma  \ref{lem:ARFL}, we routinely obtain the result.
\end{proof}

\noindent Next, we describe Theorem \ref{thm:ggee}  under the assumption that $\Phi$ has Krawtchouk type.
\begin{theorem}
\label{thm:RFLKrel}
We have
\begin{align*}
&\lbrack L, \lbrack L, F\rbrack \rbrack=0, \qquad \qquad  \lbrack R, \lbrack R, F\rbrack \rbrack=0, \\
& \lbrack F, \lbrack F, L \rbrack \rbrack - 2\lbrack L, \lbrack L, R \rbrack \rbrack = 4L, \\
& \lbrack F, \lbrack F, R \rbrack \rbrack - 2\lbrack R, \lbrack R, L \rbrack \rbrack = 4R, \\
& \lbrack F, \lbrack L, R \rbrack \rbrack=0.
\end{align*}
\end{theorem}
\begin{proof} Evaluate Theorem \ref{thm:ggee} using \eqref{eq:Kraw}, \eqref{eq:gamData}.
\end{proof}

\noindent 
Next, we describe the bijections in Lemma \ref{lem:Bij} under assumption that $\Phi$ has Krawtchouk type. We will use the
exponential function 
\begin{align*}
{\rm exp}\, z = \sum_{n =0}^\infty \frac{z^n}{n!}.
\end{align*}
\noindent Recall the maps $\mathcal R, \mathcal L$ from Definition \ref{def:RRLL}.
\begin{proposition}
\label{prop:exp} We have
\begin{align*}
\sum_{\ell =0}^d F_\ell E^*_\ell= {\rm exp} ( \mathcal L/2), \qquad \qquad \sum_{\ell=0}^d E^*_\ell F_\ell = {\rm exp} (-\mathcal L/2). 
\end{align*}
\end{proposition}
\begin{proof} Using in order Lemma \ref{lem:FE}, Proposition \ref{prop:FEform}, \eqref{eq:Kraw}, Lemma \ref{lem:cRLnil} we obtain
\begin{align*}
I &= \Biggl( \sum_{i=0}^d F_i \Biggr) \Biggl( \sum_{j=0}^d E^*_j \Biggr) 
=\sum_{j=0}^d \sum_{i=0}^j F_i E^*_j \\
&=\sum_{j=0}^d \sum_{i=0}^j \frac{\mathcal L^{j-i} }{(\theta^*_j - \theta^*_i)(\theta^*_j - \theta^*_{i+1})\cdots (\theta^*_j - \theta^*_{j-1})} F_j E^*_j 
=\sum_{j=0}^d \sum_{i=0}^j \frac{(-1)^{j-i}\mathcal L^{j-i} }{(j-i)! 2^{j-i}} F_j E^*_j  \\
&= \sum_{j=0}^d {\rm exp}(-\mathcal L/2) F_j E^*_j 
= {\rm exp}(-\mathcal L/2) \sum_{j=0}^d F_j E^*_j.
\end{align*}
The result follows from this and Lemma  \ref{lem:Bij}(iii).
\end{proof}

\noindent Next, we describe Corollary \ref{cor:c1}  under the assumption that $\Phi$ has Krawtchouk type.
\begin{theorem}
\label{cor:c1K}
We have
\begin{align*}
{\rm exp}(\mathcal L/2) R = \mathcal R \, {\rm exp}(\mathcal L/2).
\end{align*}
\end{theorem}
\begin{proof} By Corollary \ref{cor:c1} and Proposition \ref{prop:exp}.
\end{proof}

\noindent Next, we describe Corollary \ref{cor:c2} under the assumption that $\Phi$ has Krawtchouk type.
\begin{theorem}
\label{cor:c2K}
We have
\begin{align*}
{\rm exp}(\mathcal L/2) F = \Bigl(A-\mathcal R + \lbrack \mathcal L, \mathcal R\rbrack/2 \Bigr)   {\rm exp}(\mathcal L/2).
\end{align*}
\end{theorem}
\begin{proof} By Corollary \ref{cor:c2}, \eqref{eq:Kraw}, and Proposition \ref{prop:exp}.
\end{proof}

\noindent Next, we describe Corollary \ref{cor:c3} under the assumption that $\Phi$ has Krawtchouk type.
\begin{theorem}
\label{cor:c3K}
We have
\begin{align*}
{\rm exp}(\mathcal L/2) L = \Bigl(-\mathcal L + \lbrack \mathcal L, \lbrack \mathcal L, \mathcal R\rbrack \rbrack/8 \Bigr)   {\rm exp}(\mathcal L/2).
\end{align*}
\end{theorem}
\begin{proof} By Corollary \ref{cor:c3},  \eqref{eq:Kraw}, and Proposition \ref{prop:exp}.
\end{proof}

\noindent Next, we describe Theorems \ref{thm:rest}, \ref{thm:restDual}  under the assumption that $\Phi$ has Krawtchouk type.
\begin{theorem}
\label{thm:restK}
For $\ell \geq 2$ we have
\begin{align*}
({\rm ad}\, \mathcal L)^{\ell + 1} (\mathcal R) = 0, \qquad \qquad ({\rm ad}\, \mathcal R)^{\ell + 1} (\mathcal L) = 0.
\end{align*}
We are using the adjoint notation  $({\rm ad}\,X) (Y) = \lbrack X,Y\rbrack$.
\end{theorem}
\begin{proof} By Theorems \ref{thm:rest}, \ref{thm:restDual} and \eqref{eq:Kraw}.
\end{proof}

\noindent Next, we describe Corollary \ref{lem:RRRL} under the assumption that $\Phi$ as Krawtchouk type.
\begin{theorem}
\label{lem:RRRLK}
We have
\begin{align*}
\lbrack \mathcal L, \lbrack \mathcal L, \lbrack \mathcal L, \mathcal R \rbrack \rbrack \rbrack=0, \qquad \qquad
\lbrack \mathcal R, \lbrack \mathcal R, \lbrack \mathcal R, \mathcal L \rbrack \rbrack \rbrack=0.
\end{align*}
\end{theorem}
\begin{proof}
By Corollary \ref{lem:RRRL} and \eqref{eq:Kraw}.
\end{proof}

\noindent For the rest of this section, we assume that $\Phi$ is both a Leonard system and has Krawtchouk type.
The following result is well known, see for example \cite[Example~20.11]{LSnotes}.
\begin{lemma} 
\label{thm:LSK}
There exists
$p \in \mathbb F \backslash \lbrace 0, 1\rbrace$ such that
\begin{align*}
b_i = 2p(d-i), \qquad \qquad c_i = 2(1-p)i, \qquad \qquad a_i = (1-2p)(d-2i)
\end{align*} 
for $0 \leq i \leq d$ and
\begin{align*}
x_i = 4p(1-p)i(d-i+1), \qquad \qquad \varphi_i = 4p i (i-d-1)
\end{align*}
for $1 \leq i \leq d$.
\end{lemma}
 \begin{proof} This is \cite[Example~20.11]{LSnotes} with 
 \begin{align*}
 \theta_0 = d, \qquad \theta^*_0=d, \qquad s=-2, \qquad s^*=-2, \qquad r=4p.
 \end{align*}
\end{proof}
\noindent 
It is routine to check that the equations in Theorems \ref{thm:TDaixi}, \ref{thm:m4}--\ref{thm:m5} and Propositions \ref{prop:3termai}, \ref{prop:m4} are
satisfied by the values given in Lemma \ref{thm:LSK}.

\section{Acknowledgements} The author thanks Kazumasa Nomura for reading the manuscript carefully and offering valuable comments.

\bigskip


\noindent Paul Terwilliger \hfil\break
\noindent Department of Mathematics \hfil\break
\noindent University of Wisconsin \hfil\break
\noindent 480 Lincoln Drive \hfil\break
\noindent Madison, WI 53706-1388 USA \hfil\break
\noindent email: {\tt terwilli@math.wisc.edu }\hfil\break

\section{Statements and Declarations}

\noindent {\bf Funding}: The author declares that no funds, grants, or other support were received during the preparation of this manuscript.
\medskip

\noindent  {\bf Competing interests}:  The author  has no relevant financial or non-financial interests to disclose.
\medskip

\noindent {\bf Data availability}: All data generated or analyzed during this study are included in this published article.


\begin{thebibliography}{10}




\bibitem{bbit}
E.~Bannai, Et.~Bannai, T.~Ito, R.~Tanaka.
\newblock
{\em Algebraic Combinatorics.}
\newblock De Gruyter Series in Discrete Math and Applications 5.
De Gruyter, 2021.      \\
https://doi.org/10.1515/9783110630251
 







\bibitem{uniform1}
B.~Fern{\' a}ndez, R.~Maleki, {\v S}. Miklavi{\v c}, G.~Monzillo.
\newblock
Distance-regular graphs with classical parameters that support a uniform structure: case $q\leq 1$.
\newblock{\em 
Bull. Malays. Math. Sci. Soc.} 46 (2023) no. 6, Paper No. 200, 23 pp. {\tt arXiv:2305.08937}.


\bibitem{uniform2}
B.~Fern{\' a}ndez, R.~Maleki, {\v S}. Miklavi{\v c}, G.~Monzillo.
\newblock
Distance-regular graphs with classical parameters that support a uniform structure: case $q\geq 2$.
\newblock{\em 
Discrete Math.} 348 (2025) no. 1, Paper No. 114263, 12 pp. {\tt arXiv:2308.16679}.





\bibitem{twocom}
S.~Bockting-Conrad.
\newblock
Two commuting operators associated with a tridiagonal pair.
\newblock{\em
Linear Algebra Appl.
}   437  (2012) 242--270.
{\tt
arXiv:math/1110.3434}.

\bibitem{bockting2}
S.~Bockting-Conrad.
\newblock
Tridiagonal pairs of $q$-Racah type,
the double lowering operator $\psi$, and
the quantum algebra 
$U_q(\mathfrak{sl}_2)$.
\newblock{\em
Linear Algebra Appl.
}  445   (2014) 256--279.
{\tt
arXiv:1307.7410}.


\bibitem{bockTer}
S.~Bockting-Conrad and P.~Terwilliger.
\newblock
The algebra $U_q(\mathfrak{sl}_2)$  in disguise.
\newblock{\em 
Linear Algebra Appl.} 459 (2014)  548--585. {\tt arXiv:1307.7572}.

\bibitem{bockting3}
S.~Bockting-Conrad.
\newblock
Some $q$-exponential formulas involving the double lowering operator $\psi$ for a tridiagonal pair.
\newblock
Assoc. Women Math. Ser., 21
Springer, Cham, 2020, 19--43. {\tt arXiv:1907.01157}.



\bibitem{dickie}
G.~Dickie.
\newblock
Twice $Q$-polynomial distance-regular graphs are thin.
\newblock{\em 
European J. Combin.} 16 (1995)  no. 6, 555--560.

\bibitem{dickie2}
G.~Dickie and P.~Terwilliger.
\newblock
A note on thin $P$-polynomial and dual-thin $Q$-polynomial symmetric association schemes.
\newblock{\em 
J. Algebraic Combin.} 7 (1998)  5--15.

\bibitem{hartwig}
B.~Hartwig.
\newblock
The tetrahedron algebra and its finite-dimensional irreducible modules.
\newblock{\em 
Linear Algebra Appl.} 422 (2007) 219--235. {\tt arXiv:math/0606197}.

\bibitem{hora}
A.~Hora and N.~Obata.
\newblock {\rm Quantum Probability and Spectral Analysis of Graphs}.
\newblock Theoretical and Mathematical Physics, Springer, Berlin, Heidelberg 2007.

\bibitem{jtgo}
J.~Go.
\newblock
The Terwilliger algebra of the hypercube.
\newblock{\em 
European J. Combin.} 23 (2002) 399--429.

\bibitem{class}
T.~Ito, K.~Nomura, P.~Terwilliger.
\newblock
A classification of sharp tridiagonal pairs. 
\newblock{\em 
Linear Algebra Appl.} 435 (2011) 1857--1884.
{\tt arXiv:1001.1812}.

\bibitem{TD00}
T.~Ito, K.~Tanabe, P.~Terwilliger.
\newblock Some algebra related to ${P}$- and ${Q}$-polynomial association
  schemes,  in:
  \newblock {\em Codes and Association Schemes (Piscataway NJ, 1999)}, Amer.
  Math. Soc., Providence RI, 2001, pp.
       167--192.
       {\tt arXiv:math.CO/0406556}.

\bibitem{shape}
T.~Ito and P.~Terwilliger.
\newblock The shape of a tridiagonal pair.
\newblock {\em J. Pure Appl. Algebra}
	      188
	    (2004)
		     145--160.
{\tt arXiv:math.QA/0304244}.

\bibitem{tdanduqsl2hat}
T.~Ito and P.~Terwilliger.
\newblock
Tridiagonal pairs and the quantum affine algebra
 $U_q(\widehat{\mathfrak{sl}}_2)$.
\newblock{\em
Ramanujan J.}  13 (2007) 39--62.
  {\tt arXiv:math/0310042}.

\bibitem{ItoKraw}
T.~Ito and P.~Terwilliger.
\newblock Tridiagonal pairs of Krawtchouk type.
\newblock{\em 
Linear Algebra Appl.} 427 (2007) 218--233. {\tt arXiv:0706.1065}.

\bibitem{ItoTet}
T.~Ito and P.~Terwilliger.
\newblock
 Finite-dimensional irreducible modules for the three-point $\mathfrak{sl}_2$ loop algebra.
\newblock{\em 
Comm. Algebra} 36 (2008) 4557--4598. {\tt arXiv:0707.2313}.


\bibitem{TerDrinfeld}
T. Ito and P. Terwilliger. 
\newblock The Drinfeld polynomial of a tridiagonal pair.
\newblock{\em 
 J. Combin. Inform.
System Sci.} 34 (2009) 255--292. {\tt arXiv:0805.1465}.



\bibitem{tdqrac}
T.~Ito and  P.~Terwilliger.
\newblock
Tridiagonal pairs of $q$-Racah type. 
\newblock{\em J. Algebra} 322 (2009) 68--93. {\tt arXiv:0807.0271}.





 
 
\bibitem{augIto}
T.~Ito and P.~Terwilliger.
\newblock
The augmented tridiagonal algebra.
\newblock{\em  Kyushu J. Math.} 64 (2010) 8--144. {\tt arXiv:0904.2889}.




\bibitem{uniformMT}
{\v S}.~Miklavi{\v c} and P.~Terwilliger.
\newblock
Bipartite $Q$-polynomial distance-regular graphs and uniform posets.
\newblock
{\em J. Algebraic Combin.} 38 (2013)  no. 2, 225--242.
{\tt arXiv:1108.2484}.





\bibitem{nomSplit}
K.~Nomura and P.~Terwilliger.
\newblock
The split decomposition of a tridiagonal pair.
\newblock{\em 
Linear Algebra Appl.} 424 (2007)  no. 2-3, 339--345. {\tt arXiv:math/0612460}.

\bibitem{nomsharp}
K.~Nomura and P.~Terwilliger.
\newblock
Sharp tridiagonal pairs.
\newblock {\em Linear Algebra Appl.}
 429 (2008) 79--99.
{\tt arXiv:0712.3665}. 

\bibitem{nomTowards}
K.~Nomura and P.~Terwilliger.
\newblock
 Towards a classification of the tridiagonal pairs.
 \newblock {\em Linear Algebra Appl.} 429 (2008) 508--518.
{\tt arXiv:0801.0621}.

\bibitem{nomShape}
K.~Nomura and P.~Terwilliger.
\newblock
On the shape of a tridiagonal pair.
\newblock{\em
Linear Algebra Appl.} 432 (2010)  no. 2-3, 615--636.
{\tt arXiv:0906.3838}.

\bibitem{nomKraw}
K.~Nomura and P.~Terwilliger.
\newblock
Krawtchouk polynomials, the Lie algebra $\mathfrak{sl}_2$, and Leonard pairs.
\newblock{\em Linear Algebra Appl.} 437 (2012)  345--375. {\tt arXiv:1201.1645}.




\bibitem{nomTB}
K.~Nomura and P.~Terwilliger.
\newblock Totally bipartite tridiagonal pairs.
 \newblock {\em Electron. J. Linear Algebra } 37  (2021) 434--491. {\tt arXiv:1711.00332}.

\bibitem{smLP}
K.~Nomura and P.~Terwilliger.
\newblock Leonard pairs, spin models, and distance-regular graphs.
\newblock{\em 
J. Combin. Theory Ser. A} 177 (2021) Paper No. 105312, 59 pp. {\tt arXiv:1907.03900}.



\bibitem{tSub1} 
P.~Terwilliger. 
\newblock The subconstituent algebra of
an association scheme I.
\newblock{\em
J. Algebraic Combin.}
{ 1} (1992) 363--388.  

\bibitem{tSub2} 
P.~Terwilliger. 
\newblock The subconstituent algebra of
an association scheme II.
\newblock{\em
J. Algebraic Combin.}
{ 2} (1993) 73--103.  


\bibitem{tSub3} 
P.~Terwilliger. 
\newblock The subconstituent algebra of
an association scheme III. 
\newblock{ \em
J. Algebraic Combin.}
{ 2} (1993) 177--210.  


  \bibitem{qSerre}
  P.~Terwilliger.
  \newblock Two relations that generalize the $q$-Serre relations and the
  Dolan-Grady relations. In
  \newblock {\em  Physics and
  Combinatorics 1999 (Nagoya)}, 377--398, World Scientific Publishing,
   River Edge, NJ, 2001. {\tt arXiv:math/0307016}.


\bibitem{2lintrans}
P.~Terwilliger.
\newblock
Two linear transformations each tridiagonal with respect to 
an eigenbasis of the other.
\newblock{\em
Linear Algebra Appl.} 330 (2001) 149--203. {\tt arXiv:math/0406555}.
 

\bibitem{ter24}
P.~Terwilliger.
\newblock Leonard pairs from 24 points of view.
\newblock  
 Conference on Special Functions (Tempe, AZ, 2000). 
{\em  Rocky Mountain J. Math}. 32 (2002) 827--888. {\tt arXiv:math/0406577}.

\bibitem{ter2004}
P.~Terwilliger.
\newblock
Leonard pairs and the $q$-Racah polynomials.
\newblock{\em Linear Algebra Appl.} 387 (2004) 235--276. {\tt arXiv:math/0306301}.

\bibitem{terCanForm}
P.~Terwilliger.
\newblock Two linear transformations each tridiagonal with respect to
an eigenbasis of the other; the TD-D canonical form and the LB-UB 
canonical form. 
\newblock{\em J. Algebra} 291 (2005) 1--45. {\tt arXiv:math/0304077}.


\bibitem{ter2005}
P.~Terwilliger. 
\newblock
Two linear transformations each tridiagonal with respect to an 
eigenbasis of the other; comments on the parameter array.
\newblock{\em
Des. Codes Cryptogr.} 34 (2005) 307--332. {\tt arXiv:math/0306291}.

\bibitem{ter2005b}
P.~Terwilliger.
\newblock Two linear transformations each tridiagonal 
with respect to an eigenbasis of the other: 
comments on the split decomposition. 
\newblock{\em
J. Comput. Appl. Math.} 178 (2005)
 437–452. {\tt arXiv:math/0306290}.

\bibitem{madrid}
P.~Terwilliger.
\newblock An algebraic approach to the Askey scheme of orthogonal polynomials.
\newblock Orthogonal polynomials and special functions, 255--330,
Lecture Notes in Math., 1883, Springer, Berlin, 2006. {\tt arXiv:math/0408390}. 








\bibitem{LSnotes}
P.~Terwilliger.
\newblock
 Notes on the Leonard system classification.
 \newblock{\em 
Graphs Combin.} 37 (2021) no. 5, 1687--1748. {\tt arXiv:2003.09668}.

\bibitem{terAlt}
P.~Terwilliger.
\newblock
Tridiagonal pairs, alternating elements, and distance-regular graphs.
\newblock{\em 
J. Combin. Theory Ser. A } 196 (2023)  Paper No. 105724, 41 pp. {\tt arXiv:2207.07741}.

\bibitem{int}
P.~Terwilliger.
\newblock
Distance-regular graphs, the subconstituent algebra, 
and the $Q$-polynomial property.
\newblock
London Math. Soc. Lecture Note Ser., 487
Cambridge University Press, London, 2024, 430--491. {\tt arXiv:2207.07747}.

		

\bibitem{vidunas}
P.~Terwilliger and R.~Vidunas.
\newblock Leonard pairs and the Askey-Wilson relations.
\newblock{\em  J. Algebra Appl. } 3  (2004)  411--426. {\tt arXiv:math/0305356}.

\bibitem{zitnik}
P.~Terwilliger and A.~{\v Z}itnik.
\newblock
The quantum adjacency algebra and subconstituent algebra of a graph.
\newblock{\em 
J. Combin. Theory Ser. A} 166 (2019) 297--314. {\tt arXiv:1710.06011}.


\bibitem{worawannotai}
C.~Worawannotai.
\newblock
Dual polar graphs, the quantum algebra $U_q(\mathfrak{sl}_2)$, and Leonard systems of dual $q$-Krawtchouk type.
\newblock{\em 
Linear Algebra Appl.} 438 (2013) 443--497. {\tt arXiv:1205.2144}.





\end{thebibliography}
\end{document}